\documentclass[preprint,1p,10pt]{elsarticle}

\biboptions{numbers,sort&compress}

\usepackage{lineno,hyperref}
\modulolinenumbers[1]

\usepackage{array}
\newcolumntype{C}[1]{>{\centering\arraybackslash}m{#1}}

\usepackage{graphics}
\usepackage{amssymb, latexsym, amsmath, epsfig}
\usepackage{graphicx}
 \usepackage{color}
 \usepackage{epstopdf}
 \usepackage{comment}
 \usepackage{soul}
 \usepackage[normalem]{ulem}
  \usepackage{float}

\usepackage{amssymb,amsmath}
\usepackage{amsfonts,amstext,amsthm,amssymb,comment}
\usepackage{epsfig,graphics,color}

\graphicspath{{figures/}}

\usepackage{MnSymbol}

\newtheorem{theorem}{Theorem}

\numberwithin{equation}{section}

\newtheorem{remark}{Remark}

\journal{CMAME}

\begin{document}

\begin{frontmatter}

\title{A boundary penalization technique to remove outliers from isogeometric analysis on tensor-product meshes}

%\author{M. Barto$\hat{\text{n}}$\fnref{myfootnote}}

\author[adq]{Quanling Deng\corref{corr}}
\author[ad,ad2]{Victor M. Calo}
%\author[ad3]{Thomas J.R. Hughes}

%\author[ad]{Quanling Deng}

%\author[ad1]{Michael Barto\v{n}} Michael.Barton@centrum.cz, 

\cortext[corr]{Corresponding author. \\ E-mail addresses: Quanling.Deng@math.wisc.edu; Victor.Calo@curtin.edu.au}
%hughes@oden.utexas.edu}
%, vladimir.puzyrev@gmail.com, vmcalo@gmail.com.

%\author[ad]{Vladimir Puzyrev}

\address[adq]{Department of Mathematics, University of Wisconsin--Madison, Madison, WI 53706, USA}

\address[ad]{
School of Electrical Engineering, Computing and Mathematical Sciences, Curtin University, Perth, WA 6845, Australia}
%\address[ad1]{Curtin Institute for Computation, Curtin University, Kent Street, Bentley, Perth, WA 6102, Australia}

%\address[ad1]{Basque Center for Applied Mathematics, Alameda de Mazarredo 14, 48009 Bilbao, Basque Country, Spain}
\address[ad2]{Mineral Resources, Commonwealth Scientific and Industrial Research Organisation (CSIRO), Kensington, Perth, WA 6152, Australia}

%\address[ad3]{Oden Institute for Computational Engineering and Sciences, 201 East 24th Street, C0200, Austin, TX 78712-1229, USA}

\begin{abstract}
 We introduce a boundary penalization technique to improve the spectral approximation of isogeometric analysis (IGA). The method removes the outliers appearing in the high-frequency region of the approximate spectrum when using the $C^{p-1}, p$-th ($p\ge3$) order isogeometric elements. We focus on the classical Laplacian (Dirichlet) eigenvalue problem in 1D to illustrate the idea and then use the tensor-product structure to generate the stiffness and mass matrices for multidimensional problems. To remove the outliers, we penalize the higher-order derivatives from both the solution and test spaces at the domain boundary. Intuitively, we construct a better approximation by weakly imposing features of the exact solution. Effectively, we add terms to the variational formulation at the boundaries with minimal extra-computational cost. We then generalize the idea to remove the outliers for the isogeometric analysis of the Neumann eigenvalue problem (for $p\ge2$). The boundary penalization does not change the test and solution spaces. In the limiting case, when the penalty goes to infinity, we perform the dispersion analysis of $C^2$ cubic elements for the Dirichlet eigenvalue problem and $C^1$ quadratic elements for the Neumann eigenvalue problem. We obtain the analytical eigenpairs for the resulting matrix eigenvalue problems. Numerical experiments show optimal convergence rates for the eigenvalues and eigenfunctions of the discrete operator.
\end{abstract}

\begin{keyword}
 isogeometric analysis \sep spectral error  \sep eigenvalue \sep outlier \sep boundary penalization \sep dispersion analysis
%\MSC[2010] 00-01\sep  99-00
\end{keyword}

\end{frontmatter}

%\linenumbers

\section{Introduction} \label{sec:intr}

Isogeometric analysis is a Galerkin finite element method, introduced in 2005~\cite{ hughes2005isogeometric, cottrell2009isogeometric}. The method is popular and efficient.  The primary motivation for its introduction was to unify the classical finite element analysis with computer-aided design and analysis tools. Isogeometric analysis uses basis functions with higher-order continuity (smoother), for example, the B-splines or non-uniform rational basis splines (NURBS). This alternative set of basis functions improves the numerical approximations of various partial differential equations.

Cottrell et al.~\cite{ cottrell2006isogeometric} studied with isogeometric analysis structural vibrations (an eigenvalue problem) and compared the spectral approximation properties of classical finite elements against isogeometric analysis for the first time. The isogeometric elements improved the accuracy of the spectral approximation significantly in the entire discrete spectrum.  Maximum continuity basis functions remove the optical branch from the discrete spectrum of the finite element approximations. Nevertheless, a thin layer of outliers appears in the high-frequency range of the spectrum. Since then, our community has been seeking a method to remove the outliers.

There is rich literature in this line of work; see~\cite{ hughes2008duality, hughes2014finite, bartovn2018generalization, calo2019dispersion, calo2017quadrature, deng2018dispersion, puzyrev2017dispersion, puzyrev2018spectral, deng2018ddm} and references therein. In~\cite{ hughes2014finite}, the authors further explored the advantages of isogeometric analysis over finite elements on the spectral approximation. They used the Pythagorean eigenvalue error theorem (see a proof in~\cite{ strang1973analysis}) to relate the eigenvalue and eigenfunction errors. A generalization of this theorem that accounts for inexact quadrature in the inner products associated with the stiffness and mass matrices was developed in~\cite{ puzyrev2017dispersion} and this generalized theorem was applied to develop the optimally-blended quadrature rules. The work~\cite{ marfurt1984accuracy, ainsworth2010optimally} developed quadrature blending rules to reduce the dispersion errors for finite element analysis of the Helmholtz equation. In~\cite{ calo2019dispersion, puzyrev2017dispersion}, the authors generalized the idea to isogeometric elements using the dispersion analysis and reduced the spectral errors for the differential eigenvalue problems. The introduction of a single quadrature rule reduced the computational cost for the blended rules ~\cite{ deng2018dispersion} for $C^1$ quadratic isogeometric elements. The paper~\cite{ puzyrev2018spectral} studied the error reductions, stopping bands in the finite element numerical spectrum, and outliers in the isogeometric spectrum with variable continuities.  Also, error reductions for other differential operators are possible~\cite{ deng2019optimal, deng2018isogeometric}. This sequence of results delivers superconvergent eigenvalue errors. However, the outliers still pollute the high-frequencies of the spectrum, particularly in multiple dimensions.

In this paper, we propose a boundary penalization technique to remove the outliers from the isogeometric spectrum. We consider the spectral approximation of the Laplace operator. We first present the idea in 1D and then generalize the construction to 2D and 3D using tensor-product discretizations.  We penalize high-order derivatives near the boundary to remove the outliers. Intuitively, we add to the variational form near the boundary terms that build the numerical method features of the exact solution. These terms weakly impose the differential equation $-\Delta u = \lambda u$ and its higher-order derivatives, $(-\Delta)^\alpha u = \lambda^\alpha u$, at the domain boundaries. These terms assume high regularity for the eigenmodes. We only use $(-\Delta)^\alpha u = \lambda^\alpha u$ for the degrees of freedom near the boundary nodes. We impose these conditions weakly by penalization. We show that this technique removes the outliers in the numerical spectrum approximated by smooth isogeometric elements. As the penalty goes to infinity, we impose the condition $(-\Delta)^\alpha u = \lambda^\alpha u$ more strongly. In the limiting case, the penalty on $(-\Delta)^\alpha u = \lambda^\alpha u$ at the boundaries becomes an extra condition for the matrix system. For $C^2$ cubic elements for the Dirichlet eigenvalue problem and $C^1$ quadratic elements for the Neumann eigenvalue problem, based on the work~\cite{{deng2021}}, we perform the dispersion analysis (which unifies with the spectrum analysis in~\cite{ hughes2008duality}) near the boundary and give exact eigenpairs for the resulting matrix problems. For higher-order elements, a reconstruction of the basis functions near the boundaries could lead to exact eigenpairs of the matrix problems. In~\cite{Sande:2019} the authors proposed the strong imposition of higher-order derivatives for 1D problems to remove outliers, although they did not discuss the extension for multidimensional problems. However, this is not the focus of this work; we will report these results in future work.

The structure of the paper is the following.  Section~\ref{sec:ps} presents the problem and its Galerkin discretization. In Section~\ref{sec:miIgA}, we discuss how to remove the outliers for isogeometric analysis. Section~\ref{sec:nep} seeks to remove the outliers from the isogeometric analysis of the Neumann eigenvalue problem. In Section~\ref{sec:ana}, we perform the dispersion analysis for the nodes near the boundaries and give exact eigenpairs of the resulting matrix eigenvalue problems.  Section~\ref{sec:num} collects numerical examples that demonstrate the performance of the proposed ideas.  We present concluding remarks in Section~\ref{sec:conclusion}.

\section{Problem setting: Dirichlet eigenvalue problem} \label{sec:ps}

The classical second-order Dirichlet eigenvalue problem reads: find the eigenpairs $(\lambda, u)\in \mathbb{R}^+\times H^1_0(\Omega)$ with $\|u\|_\Omega=1$ such that 
\begin{equation} \label{eq:pde}
\begin{aligned}
- \Delta u & =  \lambda u \quad  \text{in} \quad \Omega, \\
u & = 0 \quad \text{on} \quad \partial \Omega,
\end{aligned}
\end{equation}
where  $\Delta = \nabla^2$ is the Laplacian, $\Omega = [0,1]^d \subset \mathbb{R}^d, d=1,2,3$, is a bounded open domain with Lipschitz boundary. The eigenvalue problem~\eqref{eq:pde} has a countable set of eigenvalues $\lambda_j \in \mathbb{R}^+$ (see, for example,~\cite[Sec. 9.8]{ Brezis:11})
\begin{equation*}
0 < \lambda_1 < \lambda_2 \leq \lambda_3 \leq \cdots
\end{equation*}
and an associated set of orthonormal eigenfunctions $u_j$, that is
\begin{equation} \label{eq:onu}
(u_j, u_k) = \delta_{jk}, 
\end{equation}
where $(\cdot, \cdot)$ denotes the $L^2-$inner product on $\Omega$ and $\delta_{lm} =1$ when $l=m$ and zero otherwise (Kronecker delta).

\subsection{Galerkin discretization}

To discretize~\eqref{eq:pde} with finite and isogeometric elements, for simplicity, we first assume that $ \Omega$ is a cube and we use a uniform tensor product mesh of size $h_x>0, h_y>0, h_z>0$ on $\Omega$. We denote each element as $E$ and its collection as $\mathcal{T}_h$ such that $\bar\Omega = \cup_{E\in \mathcal{T}_h} E$. Let $h = \max_{E\in \mathcal{T}_h} \text{diameter}(E)$.  The variational formulation of~\eqref{eq:pde} at the continuous level is to find $\lambda \in \mathbb{R}^{+}$ and $u \in H^1_0(\Omega)$ such that 
\begin{equation} \label{eq:vf}
a(w, u) =  \lambda b(w, u), \quad \forall \ w \in H^1_0(\Omega), 
\end{equation}
where $ a(w, v) = (\nabla w, \nabla v) $ and $ b(w, v) = (w, v) $.  Herein, we use the standard notation for the Hilbert and Sobolev spaces. We denote by $(\cdot,\cdot)$ and $\| \cdot \|$ the $L^2$-inner product and its norm, respectively.  From~\eqref{eq:onu}, the normalized eigenfunctions are also orthogonal in the energy inner product
\begin{equation} \label{eq:vfo}
a(u_j, u_k) =  \lambda_j b(u_j, u_k) = \lambda_j \delta_{jk}.
\end{equation}

We specify a finite dimensional approximation space $V^h_p \subset H^1_0(\Omega)$ where $V^h_p = \text{span} \{\theta^j_p\}$ is the span of the B-spline basis functions $\theta^j_p$. The definition of the B-spline basis functions in 1D is as follows.  Let $X = \{x_0, x_1, \cdots, x_m \}$ be a knot vector with knots $x_j$, which is a non-decreasing sequence of real numbers.  The $j$-th B-spline basis function of degree $p$, denoted as $\theta^j_p(x)$, is defined as~\cite{ de1978practical, piegl2012nurbs}
\begin{equation} \label{eq:Bspline}
\begin{aligned}
\theta^j_0(x) & = 
\begin{cases}
1, \quad \text{if} \ x_j \le x < x_{j+1} \\
0, \quad \text{otherwise} \\
\end{cases} \\ 
\theta^j_p(x) & = \frac{x - x_j}{x_{j+p} - x_j} \theta^j_{p-1}(x) + \frac{x_{j+p+1} - x}{x_{j+p+1} - x_{j+1}} \theta^{j+1}_{p-1}(x).
\end{aligned}
\end{equation}
The span of these basis functions generates a finite-dimensional subspace of the $H^1(\Omega)$ (see~\cite{ buffa2010isogeometric, evans2013isogeometric} for details):
\begin{equation} \label{eq:bs}
V^h_p = \text{span} \{ \Theta_j^p \}_{j=1}^{N_h} = 
\begin{cases}
 \text{span} \{ \theta^{j_x}_{p_x}(x) \}_{j_x=1}^{N_x}, & \text{in 1D}, \\
 \text{span} \{ \theta^{j_x}_{p_x}(x) \theta^{j_y}_{p_y}(y) \}_{j_x, j_y=1}^{N_x, N_y}, & \text{in 2D}, \\
 \text{span} \{\theta^{j_x}_{p_x}(x) \theta^{j_y}_{p_y}(y) \theta^{j_z}_{p_z}(z) \}_{j_x, j_y, j_z=1}^{N_x, N_y,N_z}, & \text{in 3D}, \\
\end{cases}
\end{equation}
where $p_x, p_y, p_z$ specify the approximation order in each dimension, respectively. $N_x, N_y, N_z$ are the total numbers of basis functions in each dimension, and $N_h$ is the total number of degrees of freedom.

The isogeometric analysis of~\eqref{eq:pde} seeks $\lambda^h \in \mathbb{R}^+$ and $u^h \in V^h_p$ such that 
\begin{equation} \label{eq:vfh}
a(w^h, u^h) =  \lambda^h b(w^h, u^h), \quad \forall \ w^h \in V^h_p.
\end{equation} 
We approximate the eigenfunctions as a linear combination of the B-spline basis functions and substitute all the B-spline basis functions for $w_h$ in~\eqref{eq:vfh}. This approximation leads to the following matrix eigenvalue problem
\begin{equation} \label{eq:mevp}
K U = \lambda^h M U,
\end{equation}
where $K_{kl} =  a(\theta_p^k, \theta_p^l), M_{kl} = b(\theta_p^k, \theta_p^l),$ and $U$ is the corresponding representation of the eigenvector as the coefficients of the B-spline basis functions. When we restrict the continuity of the B-splines to $C^0$, the method becomes a finite element with Bernstein basis functions.

\section{Boundary penalization technique to remove outliers} \label{sec:miIgA}

We consider $C^{p-1}, p$-th order isogeometric elements. Spectral error outliers start to appear when using $C^2$ cubic isogeometric elements for the Dirichlet eigenvalue problem. For the Neumann eigenvalue problem, they start to appear when using $C^1$ quadratic isogeometric elements.  We first describe the key idea, and then we present the idea using modified bilinear forms in 1D.

The motivating idea stems from the use of high-order numerical time-marching schemes for time-dependent problems. To illustrate the idea, we consider the wave equation $u_{tt}(x,t) - \Delta u(x,t) = f(x,t)$ with an initial solution and velocity. For a sufficiently regular solution $u$ at $x=0$, we can approximate the initial acceleration as $\Delta u(x,0) + f(x,0)$. This consistency condition allows us to compute the initial acceleration for the time integration schemes, for example, the HHT-$\alpha$~\cite{ hilber1977improved} and generalized-$\alpha$~\cite{ chung1993time} methods. Similarly, one can determine the initial higher-order time-derivative values for the heat equation by taking time-derivatives when the solution is regular enough. This class of consistency conditions is common for initializing many higher-order time-marching methods.  

Now, we borrow the idea for the eigenvalue problem~\eqref{eq:pde}. Consequently, we expect that the function $-\Delta u$ also vanishes at the boundary. Figure~\ref{fig:iga1dp3p4bf} shows the $C^2$ cubic and $C^3$ quartic splines with their second derivatives. Imposing the homogeneous Dirichlet boundary conditions removes the first and last basis functions. However, there are still two basis functions that have non-zero second derivatives at each boundary. Thus, we modify the discrete system to impose the second derivatives' homogeneity (zero value) weakly. 

\begin{figure}[ht]
\centering
\includegraphics[height=5.1cm]{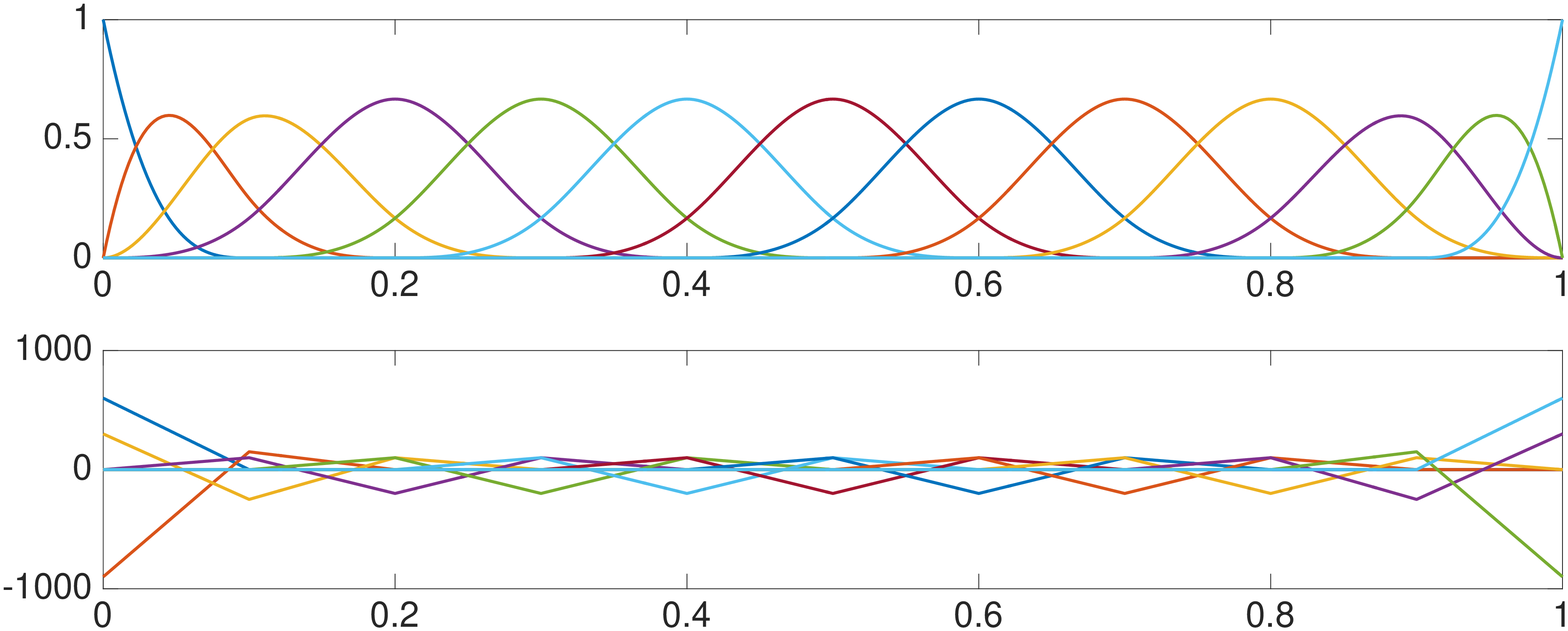} \\
\includegraphics[height=5.1cm]{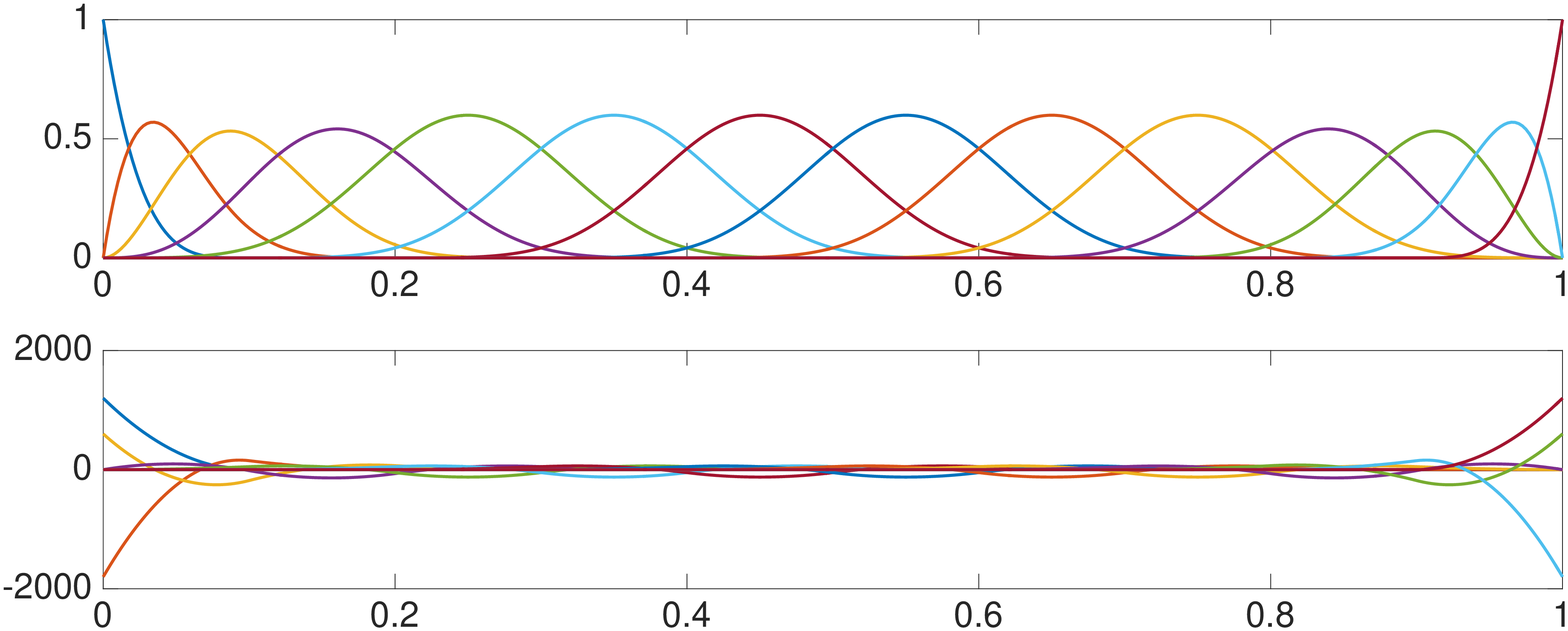} 
\vspace{-0.7cm}
\caption{$C^2$ cubic and $C^3$ quartic splines with their second derivatives.}
\label{fig:iga1dp3p4bf}
\end{figure}

For $C^2$ cubic elements, we remove the outliers from the isogeometric approximations by penalizing the Laplacian of the approximate solution at the boundaries.  For quintic $C^4$ and sextic $C^5$ isogeometric elements, there are two more outliers in the spectrum. Applying a similar argument, we obtain that the bi-Laplacian of the function also has to be zero on the boundary. Thus, we remove these additional outliers by penalizing the higher-order derivatives of the approximate solutions at the boundaries. This procedure extends to higher-order elements by following a similar construction. Herein, we remark that the extra regularity we require from the solution in these weak constraints is of the same order as the one required by operator splitting schemes that are typically used on tensor-product discretizations to accelerate the computation (see, for example,~\cite{ vabishchevich2013additive, behnoudfar2020variationally}). 

Firstly, we denote
\begin{equation}
\alpha = \lfloor \frac{p-1}{2} \rfloor =
\begin{cases}
\frac{p-1}{2}, & p \quad \text{is odd}, \\
\frac{p-2}{2}, & p \quad \text{is even}. \\
\end{cases}
\end{equation}
We present the idea in 1D. We obtain the 1D matrix problem and then apply the tensor-product structure to construct an approximate matrix problem for multidimensional problems. 

In 1D, we consider the problem~\eqref{eq:pde} on the unit interval $\Omega = [0, 1]$. We propose the boundary penalization to the isogeometric analysis of~\eqref{eq:pde} in 1D as: Find $\tilde \lambda^h \in \mathbb{R}$ and $\tilde u^h \in V^h_p$ such that 
\begin{equation} \label{eq:dcvfh}
  \tilde a(w^h, \tilde u^h) =  \tilde\lambda^h b(w^h, \tilde u^h), \quad \forall \ w^h \in V^h_p,
\end{equation} 
where for $w, v \in V^h_p$
\begin{subequations} \label{eq:dcvfbfs}
  \begin{align}
    \tilde a(w,  v) & = \int_0^1 w' v' \ d x + \sum_{\ell = 1}^\alpha \eta_{a,\ell} \pi^2 h^{6\ell-3} s^{(2\ell)}(w,  v), \label{eq:dcvfbfsa} \\
    \tilde b(w,  v) & = \int_0^1 w v \ d x + \sum_{\ell = 1}^\alpha \eta_{b,\ell} h^{6\ell-1} s^{(2\ell)}(w,  v), \label{eq:dcvfbfsb} \\
    s^{(2\ell)}(w,  v) & = w^{(2\ell)}(0) v^{(2\ell)}(0) + w^{(2\ell)}(1) v^{(2\ell)}(1), \label{eq:dcvfbfsc}
  \end{align}
\end{subequations}
and $\eta_{a,\ell}, \eta_{b,\ell}$ are parameters that penalize the (jumps of) product of higher-order derivatives at the domain boundaries. As default we set $\eta_{a,\ell} = \eta_{b,\ell} = 1$. We refer to this penalty technique near the boundaries as a discrete correction (DC) technique while the overall method as a discretely corrected isogeometric analysis (DC-IGA).

\begin{remark} \label{rem:sc}
For the bilinear form $\tilde a(\cdot, \cdot)$, we set the scaling $h^{6\ell-3}$, where $6\ell - 3 = (4\ell - 1) + 2(\ell - 1)$. The part $(4\ell - 1)$ contributes to the scaling that is consistent the scaling of the original inner-product of the derivatives $(w', v')$. The part $2(\ell - 1)$ contributes to the consistency errors from the $(-\Delta)^\ell u = \lambda^\ell u$. As a consequence, we rewrite $\lambda^\ell = \lambda \lambda^{\ell-1}$, where $\lambda^{\ell-1}$ scales like $\mathcal{O}(h^{2(1-\ell)})$ in the high-frequency region. Thus, the bilinear form $\tilde b(\cdot, \cdot)$ has a scaling with two extra orders with respect to $h$. We validate this observation in the numerical section. 
\end{remark}

As in~\eqref{eq:mevp}, this leads to the matrix eigenvalue problem
\begin{equation} \label{eq:mevp1d}
\tilde{K}^{1D} U = \tilde \lambda^h \tilde{M}^{1D} U,
\end{equation}
where $\tilde{K}_{kl} =  \tilde a(\theta_p^k, \theta_p^l), \tilde{M}_{kl} = \tilde b(\theta_p^k, \theta_p^l),$ and $U$ is the corresponding representation of the eigenvector as the coefficients of the B-spline basis functions. Using the tensor-product structure and introducing the outer product, the corresponding 2D matrix problem is
\begin{equation} \label{eq:mevp2d}
( \tilde{K}^{1D}_x \otimes \tilde{M}^{1D}_y + \tilde{M}^{1D}_x \otimes \tilde{K}^{1D}_y )  U = \tilde \lambda^h \tilde{M}^{1D}_x \otimes \tilde{M}^{1D}_y U,
\end{equation}
where $\otimes$ is the Kronecker product, and 3D matrix problem is
\begin{equation} \label{eq:mevp3d}
( \tilde{K}^{1D}_x \otimes \tilde{M}^{1D}_y \otimes \tilde{M}^{1D}_z + \tilde{M}^{1D}_x \otimes \tilde{K}^{1D}_y \otimes \tilde{M}^{1D}_z + \tilde{M}^{1D}_x \otimes \tilde{M}^{1D}_y \otimes \tilde{K}^{1D}_z )  U = \tilde \lambda^h \tilde{M}^{1D}_x \otimes \tilde{M}^{1D}_y \otimes \tilde{M}^{1D}_z U,
\end{equation}
where $\tilde{K}^{1D}_q,  \tilde{M}^{1D}_q, q=x,y,z$ are 1D matrices generated as in~\eqref{eq:mevp1d} from the modified bilinear forms in~\eqref{eq:dcvfbfs}. We refer to~\cite{ ainsworth2010optimally, gao2013kronecker} for similar extensions to multiple dimensions. 

Alternatively, for multidimensional problem~\eqref{eq:pde}, to remove the outliers, one can penalize the boundary nodes similarly with the following bilinear forms
\begin{subequations} 
  \begin{align}
    \tilde a(w,  v) & = (\nabla w, \nabla v) + \sum_{\ell = 1}^\alpha \eta_{a,\ell} \pi^2 h^{6\ell-3} s^{(2\ell)}(w,  v), \\
    \tilde b(w,  v) & = ( w,  v ) + \sum_{\ell = 1}^\alpha \eta_{b,\ell} h^{6\ell-1} s^{(2\ell)}(w,  v), \\
    s^{(2\ell)}(w,  v) & = (\Delta^\ell w, \Delta^\ell v)_{\partial \Omega},
  \end{align}
\end{subequations}
where $(\cdot, \cdot)_{\partial \Omega}$ denotes the $L^2$ inner product on the boundary ${\partial \Omega}$. These bilinear forms lead to matrix eigenvalue problems which are slightly different (i.e., modified entries near the boundaries) from the problems~\eqref{eq:mevp2d} in 2D and~\eqref{eq:mevp3d} in 3D. The advantage is that on the penalty terms, one does not require a tensor-product mesh (but may require on the original isogeometric basis functions). In the numerical simulation, we observe that these bilinear forms remove the outliers. However, the spectral errors at high-frequency region are larger. Thus, for multidimensional problem~\eqref{eq:pde}, we focus on the~\eqref{eq:mevp2d} in 2D and~\eqref{eq:mevp3d} in 3D.

\begin{remark}
  The new scheme corresponds to a consistent boundary penalization. The correction is consistent in the sense that the extra terms vanish for smooth solutions. For $p=1,2$, $\alpha=0$, thus, the extra terms become zeros. There are no outliers and hence no modification to the bilinear forms. The additional terms penalize the jumps of the spatial operators at the boundary nodes.  The jump terms consider the differences of the $2\ell$-th derivatives of the basis functions and the Dirichlet boundary values (zeros, in this case, hence omitted in the modified bilinear forms).
\end{remark}

Our numerical results show that for $\ell \ge 1$, we can set $\eta_{a,\ell}=0$ or $\eta_{b,\ell}=0$ and assign the other parameter appropriately to remove the outliers from the spectrum. However, the eigenfunction errors remain large. When $\eta_{a,\ell}>0$ or $\eta_{b,\ell}>0$, the eigenfunction errors decrease in the high-frequency region.  By default, we set $\eta_{a,\ell}= \eta_{b,\ell} =1$ for all $\ell$. The modifications to the variational forms are local to the basis functions with support on the elements contiguous to the boundary. With these modifications, we observe optimally convergent errors in both eigenvalues and eigenfunctions.

\section{Neumann eigenvalue problem} \label{sec:nep}

Now, we consider the Neumann eigenvalue problem
\begin{equation} \label{eq:nep}
\begin{aligned}
- \Delta u & =  \lambda u \quad  \text{in} \quad \Omega, \\
\nabla u \cdot \boldsymbol{n} & = 0 \quad \text{on} \quad \partial \Omega.
\end{aligned}
\end{equation}
For this problem, the outliers appear in the spectrum of isogeometric elements when using $C^1$ quadratic elements. We describe the bilinear forms for the 1D case to obtain the matrix problem~\eqref{eq:mevp1d}  that remove the outliers from the discrete solution. 

We denote 
\begin{equation}
\beta = \lfloor \frac{p}{2} \rfloor =
\begin{cases}
\frac{p}{2}, & p \quad \text{is even}, \\
\frac{p-1}{2}, & p \quad \text{is odd}. \\
\end{cases}
\end{equation} 
Extending the arguments we propose for the Dirichlet eigenvalue problem~\eqref{eq:pde}, we introduce a boundary penalization to the isogeometric discretization of~\eqref{eq:nep} in 1D as: Find $\hat \lambda^h \in \mathbb{R}$ and $\hat u^h \in V^h_p$ such that
\begin{equation} \label{eq:dcvfhnep}
\hat a(w^h, \hat u^h) =  \hat\lambda^h b(w^h, \hat u^h), \quad \forall \ w^h \in V^h_p,
\end{equation} 
where for $w, v \in V^h_p$
\begin{subequations} \label{eq:dcvfbfsnep}
  \begin{align}
    \hat a(w,  v) & = \int_0^1 w' v' \ d x + \sum_{\ell = 1}^\beta \hat\eta_{a,\ell} \pi^2 h^{6\ell-5} s^{(2\ell-1)}(w,  v), \label{eq:dcvfbfsnepa} \\
    \hat b(w,  v) & = \int_0^1 w v \ d x + \sum_{\ell = 1}^\beta \hat \eta_{b,\ell} h^{6\ell-3} s^{(2\ell-1)}(w,  v), \label{eq:dcvfbfsnepb} \\
    s^{(2\ell-1)}(w,  v) & = w^{(2\ell-1)}(0) v^{(2\ell-1)}(0) + w^{(2\ell-1)}(1) v^{(2\ell-1)}(1), \label{eq:dcvfbfsnepc}
  \end{align}
\end{subequations}
and $\hat\eta_{a,\ell}, \hat\eta_{b,\ell}$ penalize the jumps. By default, we set $\hat\eta_{a,\ell} = \hat\eta_{b,\ell} = 1$. We use the scaling orders by following the arguments in Remark~\ref{rem:sc} for the Dirichlet eigenvalue problem. Similarly, for multidimensional problem~\eqref{eq:nep}, we use~\eqref{eq:mevp2d} and~\eqref{eq:mevp3d} for 2D and 3D matrix problems, respectively.

\section{Infinite penalization and boundary dispersion analysis} \label{sec:ana}

In this section, we consider the limiting case when the penalty goes to $\infty$. In such a case, the penalized terms on $(-\Delta)^\alpha u = \lambda^\alpha u$ at the boundaries become strong boundary conditions for the resulting matrix problems. In particular, for the Dirichlet eigenvalue problem~\eqref{eq:pde}, we present the dispersion analysis near the boundaries for $C^2$ cubic isogeometric elements. For the Neumann eigenvalue problem~\eqref{eq:nep}, we present the dispersion analysis near the boundaries for $C^1$ quadratic isogeometric elements. In these two cases on uniform meshes, we obtain the exact eigenvalues and eigenvectors from the resulting matrix eigenvalue problems.  To give the exact eigenpairs, we first present the following result.
\begin{theorem}[Analytical eigenvalues and eigenvectors, case 1] \label{thm:evab1}
Let $G^{(\xi)} = (G^{(\xi)}_{j,k})$ be a square matrix with dimension $n$ such that 
\begin{equation} \label{eq:G}
G^{(\xi)}_{j,j+k} = \xi_{|k|}, \quad  |k| \le m, \quad  j = 1,\cdots, n.
\end{equation}
Let $H^{(\xi)} = (H^{(\xi)}_{j,k})$ be a zero square matrix with dimension $n$. For $m\ge 2$, we modify $H^{(\xi)}$ with
\begin{equation} \label{eq:H1}
H^{(\xi)}_{j,k} = H^{(\xi)}_{n-j+1,n-k+1}=  \xi_{j+k}, \quad  k = 1, \cdots, m-j, \quad j =1, \cdots, m-1.
\end{equation}
Let 
$
A = G^{(\mu)} - H^{(\mu)}, B = G^{(\nu)} - H^{(\nu)}.
$
Assume that $B$ is invertible. Then, the generalized matrix eigenvalue problem $A X = \lambda BX$ has eigenpairs $(\lambda_j, X_j)$ with $X_j = (X_{j,1}, \cdots, X_{j, n})^T$ where
\begin{equation} \label{eq:set1}
  \lambda_j = \frac{ \mu_0 + 2 \sum_{l=1}^{m} \mu_l \cos(l j\pi h) }{ \nu_0 + 2 \sum_{l=1}^{m} \nu_l \cos(l j\pi h) }, \quad X_{j,k} = c \sin( j \pi k h), \quad h = \frac{1}{n+1}, \  j, k =1,2,\cdots, n,
\end{equation}
where $0 \ne c \in \mathbb{R}.$
\end{theorem}

\begin{proof}
  The matrices defined here are Toeplitz matrices with modified entries near the boundaries.  From Bloch wave assumption and with the modifications near the boundaries in mind, we seeks eigenvectors of the form $\sin( j \pi k h)$. Using the trigonometric identity $\sin(\alpha \pm \beta) = \sin(\alpha) \cos(\beta) \pm \cos(\alpha) \sin(\beta)$, one can verify that each row of the problem $A X = \lambda BX$, i.e., $\sum_{k=1}^n A_{ik} X_{j,k} = \lambda \sum_{k=1}^n B_{ik} X_{j,k}, i=1,\cdots, n$, reduces to $\alpha_0 + 2 \sum_{l=1}^{m} \alpha_l \cos(l j\pi h) = \lambda \big( \beta_0 + 2 \sum_{l=1}^{m} \beta_l \cos(l j\pi h) \big)$. This is independent of the row number $i$, that is, all the boundary and internal rows lead to the same expression. The problem has at most $n$ eigenpairs and the $n$ eigenvectors are linearly independent.  Thus, the $n$ eigenpairs $(\lambda_j, X_j)$ given in~\eqref{eq:set1} are the analytical solutions to the problem $A X = \lambda BX$.  This completes the proof.
\end{proof}

The eigenpairs given in this theorem correspond to the Dirichlet eigenvalue problem.  Similarly, we have the following result for the Neumann eigenvalue problem.  We omit the proof for brevity.
\begin{theorem}[Analytical eigenvalues and eigenvectors, case 2] \label{thm:evab2}
Let $G^{(\xi)}$ be defined in~\eqref{eq:G}. 
Let $H^{(\xi)} = (H^{(\xi)}_{j,k})$ be a zero square matrix with dimension $n$. We modify $H^{(\xi)}$ with
\begin{equation} \label{eq:H2}
  H^{(\xi)}_{j,k} = H^{(\xi)}_{n-j+1,n-k+1}=  \xi_{j+k-1}, \quad  k = 1, \cdots, m-j+1, \quad j =1, \cdots, m.
\end{equation}
Let
$
A = G^{(\mu)} + H^{(\mu)}, B = G^{(\nu)} + H^{(\nu)}.
$
Assume that $B$ is invertible.  Let $h = \frac{1}{n}$. Then, the generalized matrix eigenvalue problem $A X = \lambda BX$ has eigenpairs $(\lambda_j, X_j)$ with $X_j = (X_{j,1}, \cdots, X_{j, n})^T$ where
\begin{equation}  \label{eq:set3}
  \lambda_{j+1} = \frac{ \mu_0 + 2 \sum_{l=1}^{m} \mu_l \cos(l j\pi h) }{ \nu_0 + 2 \sum_{l=1}^{m} \nu_l \cos(l j\pi h) }, \qquad X_{j+1,k} = c \cos\big( j \pi (k-\frac12) h \big)
\end{equation}
with $k =1,2,\cdots, n, \ j=0,1,\cdots,n-1, 0 \ne c \in \mathbb{R}.$
\end{theorem}

The matrix $G$ in~\eqref{eq:G} is a Toeplitz matrix while the matrices $H$ defined in~\eqref{eq:H1} and~\eqref{eq:H2} are zero matrices with modified entries at the first and last few rows (near boundaries). All these matrices are symmetric and persymmetric. We refer to~\cite[Section 2]{ deng2021} for more details.

\subsection{$C^2$ cubic elements for the Dirichlet eigenvalue problem}

First, we solve the problem~\eqref{eq:pde} in 1D by the isogeometric elements~\eqref{eq:vfh} with $C^2$ cubic B-splines. We discretize the unit interval $\Omega = [0,1]$ with a uniform grid $\mathcal{T}_h = \{ x_j = jh \}_{j=0}^N$, where $h = 1/N$ is the grid size. The $C^2$ cubic isogeometric elements lead to the stiffness and mass matrices (after applying the homogeneous Dirichlet boundary conditions)
\begin{equation} \label{eq:km3}
  \begin{aligned}
    K & = \frac{1}{h}
    \begin{bmatrix}
      \frac{3}{2} & \frac{3}{80} & -\frac{1}{4} & -\frac{1}{80} \\[0.2cm]
      \frac{3}{80} & \frac{27}{40} & -\frac{1}{30} & -\frac{47}{240} & -\frac{1}{120} \\[0.2cm]
      -\frac{1}{4} & -\frac{1}{30} & \frac{2}{3} & -\frac{1}{8} & -\frac{1}{5} & -\frac{1}{120}  \\[0.2cm]
      -\frac{1}{80} & -\frac{47}{240} & -\frac{1}{8} & \frac{2}{3} & -\frac{1}{8} & -\frac{1}{5} & -\frac{1}{120}  \\[0.2cm]
      & -\frac{1}{120} & -\frac{1}{5} & -\frac{1}{8} & \frac{2}{3} & -\frac{1}{8} & -\frac{1}{5} & -\frac{1}{120}  \\[0.2cm]
      &  & \ddots & \ddots & \ddots & \ddots & \ddots & \ddots  & \ddots  \\[0.2cm]
      % &&& -\frac{1}{120} & -\frac{1}{5} & -\frac{1}{8} & \frac{2}{3} & -\frac{1}{8} & -\frac{1}{5} & -\frac{1}{120} \\[0.2cm]
      & & & -\frac{1}{120}  & -\frac{1}{5} & -\frac{1}{8} & \frac{2}{3} & -\frac{1}{8} & -\frac{47}{240}  & -\frac{1}{80}  \\[0.2cm]
      & & & & -\frac{1}{120}  & -\frac{1}{5} & -\frac{1}{8} & \frac{2}{3} & -\frac{1}{30} &  -\frac{1}{4}  \\[0.2cm]
      & && & & -\frac{1}{120}  & -\frac{47}{240} & -\frac{1}{30} & \frac{27}{40} & \frac{3}{80}   \\[0.2cm]
      & & && & & -\frac{1}{80}  & -\frac{1}{4} & \frac{3}{80} & \frac{3}{2}   \\[0.2cm]
    \end{bmatrix}, \\
    M &= h
    \begin{bmatrix}
      \frac{31}{140} & \frac{5}{32} & \frac{29}{840} & \frac{1}{3360} \\[0.2cm]
      \frac{5}{32} & \frac{183}{560} & \frac{283}{1260} & \frac{239}{10080} & \frac{1}{5040} \\[0.2cm]
      \frac{29}{840} &\frac{283}{1260} & \frac{151}{315} & \frac{397}{1680} & \frac{1}{42} & \frac{1}{5040}  \\[0.2cm]
      \frac{1}{3360} & \frac{239}{10080}  & \frac{397}{1680} & \frac{151}{315} &\frac{397}{1680} & \frac{1}{42} & \frac{1}{5040}  \\[0.2cm]
      & \frac{1}{5040}  & \frac{1}{42} & \frac{397}{1680} & \frac{151}{315} &\frac{397}{1680} & \frac{1}{42} & \frac{1}{5040}  \\[0.2cm]
      &  & \ddots & \ddots & \ddots & \ddots & \ddots & \ddots  & \ddots  \\[0.2cm]
      & & & \frac{1}{5040}  & \frac{1}{42} & \frac{397}{1680} & \frac{151}{315} & \frac{397}{1680} & \frac{239}{10080}  & \frac{1}{3360} \\[0.2cm]
      & & & & \frac{1}{5040}  & \frac{1}{42} & \frac{397}{1680} & \frac{151}{315} & \frac{283}{1260} & \frac{29}{840} \\[0.2cm]
      & &&& & \frac{1}{5040}  & \frac{239}{10080}  & \frac{283}{1260}  & \frac{183}{560}  & \frac{5}{32} \\[0.2cm]
      & & &&& & \frac{1}{3360}  & \frac{29}{840}   & \frac{5}{32} & \frac{31}{140}  \\[0.2cm]
    \end{bmatrix}, \\
  \end{aligned}
\end{equation}
which are of dimension $(N+1) \times (N+1)$.

For the purpose of dispersion analysis, one assumes that the component of the eigenvector $U_j$ of~\eqref{eq:mevp} takes the form (see, e.g.,~\cite[Sec. 5]{ cottrell2006isogeometric} and~\cite[Sec. 4]{ hughes2008duality})
\begin{equation} \label{eq:ujk}
  U_{j,k} = \sin(\omega_j k h), \quad k = 1, 2, \cdots, N+1,
\end{equation}
where $\omega_j = j\pi$ is the eigenfrequency associated with the $j$-th mode. For an internal node with $5 \le k \le N-3$, the $k$-th row of the matrix problem~\eqref{eq:mevp} with $K$ and $M$ defined in~\eqref{eq:km3} leads to the dispersion relation
\begin{equation} \label{eq:drep3}
  \lambda^h_j h^2 = - 42  + \frac{1008 \big(52 + 49 \cos(\omega_j h) + 4 \cos(2 \omega_j h) \big)}{1208 + 1191 \cos(\omega_j h) + 120 \cos(2 \omega_j h) +  \cos(3 \omega_j h)}.
\end{equation}
With $\omega_j h < 1$, a Taylor expansion of~\eqref{eq:drep3} leads to the dispersion error representation
\begin{equation} \label{eq:tep3}
\Lambda^h = \Lambda + \frac{1}{30240} \Lambda^4 + \mathcal{O}(\Lambda^5),
\end{equation}
where $\Lambda = (\omega_j h)^2$ and $\Lambda^h = h^2 \lambda^h_j$. The dispersion error is given by subtracting both sides $\Lambda$ and dividing both sides by $h^2$. Similarly, the first four (similarly, for the last four) rows of matrix problem~\eqref{eq:mevp} with $K$ and $M$ defined in~\eqref{eq:km3} lead to the dispersion relations
\begin{equation} \label{eq:drep4}
  \begin{aligned}
    \lambda^h_j h^2 & = - 42  + \frac{2016 \big(10 + 11 \cos(\omega_j h) + 2 \cos(2 \omega_j h) \big)}{430 + 526 \cos(\omega_j h) + 116 \cos(2 \omega_j h) +  \cos(3 \omega_j h)}, \\
    \lambda^h_j h^2 & = - 42  + \frac{4032 \big(40 + 76 \cos(\omega_j h) + 47 \cos(2 \omega_j h) + 4\cos(3 \omega_j h) \big)}{3841 + 7066 \cos(\omega_j h) + 4532 \cos(2 \omega_j h) +  478\cos(3 \omega_j h) + 4\cos(4 \omega_j h)}, \\
    \lambda^h_j h^2 & = - 42  + \frac{1008 \big(57 + 96 \cos(\omega_j h) + 108 \cos(2 \omega_j h) + \cdots \big)}{1355 + 2324 \cos(\omega_j h) + 2536 \cos(2 \omega_j h) + \cdots}, \\
    \lambda^h_j h^2 & = - 42  + \frac{84 \big(96 \sin(2 \omega_j h) + 1176 \sin(3 \omega_j h) + \cdots \big)}{3 \sin(\omega_j h) + 239 \sin(2 \omega_j h) +\cdots},
  \end{aligned}
\end{equation}
which have the following error representations  
\begin{equation} \label{eq:dreep4}
  \begin{aligned}
    \Lambda^h &=  \frac{1302}{1073} + \mathcal{O}(\Lambda),\\
    \Lambda^h &=  \frac{518}{1769} + \mathcal{O}(\Lambda), \\
    \Lambda^h &=  \frac{42}{941} + \mathcal{O}(\Lambda), \\
    \Lambda^h &=  \frac{14}{13439} + \mathcal{O}(\Lambda),
  \end{aligned}
\end{equation}
respectively. These dispersion relations and errors near the boundary are inconsistent with~\eqref{eq:drep3} and~\eqref{eq:tep3} for the internal nodes. This inconsistency leads to the outliers in the approximate spectrum. The boundary penalization technique proposed in Section~\ref{sec:miIgA} recovers the consistency weakly. In the limiting case, when the penalty goes to infinity for $C^2$ cubic elements, we recover the consistency strongly. We perform the analysis as follows.

In the limiting case when the penalty goes to $\infty$, the penalized terms in~\eqref{eq:dcvfbfs} at the boundaries $x=0$ and $x=1$ become the conditions on the unknown vector components of the generalized matrix eigenvalue problem~\eqref{eq:mevp}
\begin{equation} \label{eq:3}
3 U_1 - U_2 = 0, \qquad 3 U_{N+1} - U_N = 0.
\end{equation}
We apply these conditions to reduce the matrix eigenvalue problem~\eqref{eq:mevp} from dimension $(N+1) \times (N+1)$ to dimension $(N-1) \times (N-1)$. In particular, we multiply the first column by $\frac13$ and add the result to the second column, then multiply the first row by $\frac13$ and add the result to the second row. Similarly, we multiply the $(N+1)$-th column by $\frac13$ and add the result to the $N$-th column, then multiply the $(N+1)$-th row by $\frac13$ and add the result to the $N$-th row. Removing the first and last rows and columns, we obtain
\begin{equation} \label{eq:km3new}
  \begin{aligned}
    \tilde{K} & = \frac{1}{h}
    \begin{bmatrix}
      \frac{13}{15} & -\frac{7}{60} & -\frac{1}{5} & -\frac{1}{120} \\[0.2cm]
      -\frac{7}{60} & \frac{2}{3} & -\frac{1}{8} & -\frac{1}{5} & -\frac{1}{120} \\[0.2cm]
      -\frac{1}{5} & -\frac{1}{8} & \frac{2}{3} & -\frac{1}{8} & -\frac{1}{5} & -\frac{1}{120}  \\[0.2cm]
      -\frac{1}{120} & -\frac{1}{5} & -\frac{1}{8} & \frac{2}{3} & -\frac{1}{8} & -\frac{1}{5} & -\frac{1}{120}  \\[0.2cm]
      & \ddots & \ddots & \ddots & \ddots & \ddots & \ddots  & \ddots  \\[0.2cm]
      & & -\frac{1}{120}  & -\frac{1}{5} & -\frac{1}{8} & \frac{2}{3} & -\frac{1}{8} & -\frac{1}{5}   \\[0.2cm]
      & & & -\frac{1}{120}  & -\frac{1}{5} & -\frac{1}{8} & \frac{2}{3} & -\frac{7}{60}   \\[0.2cm]
      && & & -\frac{1}{120}  & -\frac{1}{5} & -\frac{7}{60} & \frac{13}{15}   \\[0.2cm]
    \end{bmatrix}, \\
    \tilde{M} &= h
    \begin{bmatrix}
      \frac{41}{90} & \frac{17}{72} & \frac{1}{42} & \frac{1}{5040} \\[0.2cm]
      \frac{17}{72} & \frac{151}{315} & \frac{397}{1680} & \frac{1}{42} & \frac{1}{5040} \\[0.2cm]
      \frac{1}{42} &\frac{397}{1680} & \frac{151}{315} & \frac{397}{1680} & \frac{1}{42} & \frac{1}{5040}  \\[0.2cm]
      \frac{1}{5040} & \frac{1}{42} & \frac{397}{1680} & \frac{151}{315} &\frac{397}{1680} & \frac{1}{42} & \frac{1}{5040}  \\[0.2cm]
      & \ddots & \ddots & \ddots & \ddots & \ddots & \ddots  & \ddots  \\[0.2cm]
      & & \frac{1}{5040}  & \frac{1}{42} & \frac{397}{1680} & \frac{151}{315} & \frac{397}{1680} & \frac{1}{42} \\[0.2cm]
      & & & \frac{1}{5040}  & \frac{1}{42} & \frac{397}{1680} & \frac{151}{315} & \frac{17}{72} \\[0.2cm]
      &&& & \frac{1}{5040}  & \frac{1}{42} & \frac{17}{72} & \frac{41}{90}  \\[0.2cm]
    \end{bmatrix}, \\
  \end{aligned}
\end{equation}
which are of dimensions $(N-1) \times (N-1)$. Applying Theorem~\ref{thm:evab1}, we have the following result.

\begin{theorem}[Analytical solutions] \label{lem:p3}
  Let $\tilde{K}$ and $\tilde{M}$ be defined in~\eqref{eq:km3new}.  The problem $\tilde K U = \lambda^h \tilde M U$ has analytical eigenpairs $(\lambda_j^h, U_j )$ where
\begin{equation} \label{eq:evefp3}
  \lambda^h_j = - 42N^2  + \frac{1008N^2 \big(52 + 49 \cos(j \pi h) + 4 \cos(2 j \pi h) \big)}{1208 + 1191 \cos(j \pi h) + 120 \cos(2 j \pi h) +  \cos(3 j \pi h)}, \quad U_{j,k} = c \sin(j\pi k h)
\end{equation}
with $h = \frac{1}{N}, j, k =1,2,\cdots, N-1, 0 \ne c \in \mathbb{R}$. 
\end{theorem}
 
\begin{remark}
For the matrix eigenvalue problem with matrices $\tilde{K}$ and $\tilde{M}$ as defined in~\eqref{eq:km3new}, the corresponding dispersion relations for both the internal and boundary nodes are the same as in~\eqref{eq:drep3}. That is, the dispersion relations and errors are consistent for all the eigenvector components. In such a case, the analytical approximate eigenvalues are given in~\eqref{eq:evefp3}, and it is a smooth function in terms of the index $j$; therefore, no outliers. For multidimensional problems, the analytical eigenpairs can be obtained by applying Theorem 2.8 in~\cite{deng2021}. For higher-order elements, one requires a reconstruction of the basis functions near the boundaries to recover this consistency. We will report these results shortly.
\end{remark}

\subsection{$C^1$ quadratic elements for the Neumann eigenvalue problem}

We now present the boundary dispersion analysis for the Neumann eigenvalue problem~\eqref{eq:nep} when using $C^1$ quadratic elements. On a uniform grid with $N$ elements in 1D, the isogeometric analysis with $C^1$ quadratic elements for~\eqref{eq:nep} leads to the following stiffness and mass matrices
\begin{equation} \label{eq:km2}
  \begin{aligned}
    K & = \frac{1}{h}
    \begin{bmatrix}
      \frac{4}{3} & -1 & -\frac{1}{3} \\[0.2cm]
      -1 & \frac{4}{3} & -\frac{1}{6} & -\frac{1}{6} \\[0.2cm]
      -\frac{1}{3} & -\frac{1}{6} & 1 & -\frac{1}{3} & -\frac{1}{6} \\[0.2cm]
      &   -\frac{1}{6} & -\frac{1}{3} & 1 & -\frac{1}{3} & -\frac{1}{6} \\[0.2cm]
      &   & \ddots & \ddots & \ddots & \ddots & \ddots &  \\[0.2cm]
      &  & & -\frac{1}{6} & -\frac{1}{3} & 1 & -\frac{1}{3} & -\frac{1}{6} \\[0.2cm]
      & & &  & -\frac{1}{6} & -\frac{1}{3} & 1 & -\frac{1}{6} & -\frac{1}{3}   \\[0.2cm]
      & &&& & -\frac{1}{6} & -\frac{1}{6} & \frac{4}{3}  & -1 \\[0.2cm]
      & & &&& & -\frac{1}{3} & -1 & \frac{4}{3}  \\
    \end{bmatrix}, \\
    M &= h
    \begin{bmatrix}
      \frac{1}{5} & \frac{7}{60} & \frac{1}{60} \\[0.2cm]
      \frac{7}{60} & \frac{1}{3} & \frac{5}{24} & \frac{1}{120} \\[0.2cm]
      \frac{1}{60} & \frac{5}{24} & \frac{11}{20} & \frac{13}{60} & \frac{1}{120} \\[0.2cm]
      & \frac{1}{120} & \frac{13}{60} & \frac{11}{20} & \frac{13}{60} & \frac{1}{120} \\[0.2cm]
      &  & \ddots & \ddots & \ddots & \ddots & \ddots &  \\[0.2cm]
      &  & & \frac{1}{120} & \frac{13}{60} & \frac{11}{20} & \frac{13}{60} & \frac{1}{120} \\[0.2cm]
      &  & &  & \frac{1}{120} & \frac{13}{60} & \frac{11}{20} & \frac{5}{24} & \frac{1}{60} \\[0.2cm]
      & &&& & \frac{1}{120} & \frac{5}{24}  & \frac{1}{3}  & \frac{7}{60} \\[0.2cm]
      & & &&& & \frac{1}{60} & \frac{7}{60}  & \frac{1}{5}  
    \end{bmatrix},
  \end{aligned}
\end{equation}
which are of dimensions $(N+2) \times (N+2)$. 

Similarly, we assume that the internal component of the eigenvector $U_j$ of~\eqref{eq:mevp} with these matrices takes the form 
\begin{equation} \label{eq:ujknep}
U_{j,k} = \cos(\omega_j \big(k-\frac32) h \big), \quad k = 1, 2, \cdots, N+2,
\end{equation}
where $\omega_j = j\pi$ is the eigenfrequency associated with the $j$-th mode. For an internal node with $4 \le k \le N-1$, the  $k$-th row of the matrix problem~\eqref{eq:mevp} with $K$ and $M$ defined in~\eqref{eq:km2} leads to the dispersion relation
\begin{equation} \label{eq:drep2nep}
  \lambda^h_j h^2 = - 20  + \frac{240 \big(3 + 2 \cos(\omega_j h) \big)}{33 + 26 \cos(\omega_j h) + \cos(2 \omega_j h) }.
\end{equation}
With $\omega_j h < 1$, a Taylor expansion of~\eqref{eq:drep2nep} leads to the dispersion error representation
\begin{equation} \label{eq:tep2nep}
\Lambda^h = \Lambda + \frac{1}{720} \Lambda^3 + \mathcal{O}(\Lambda^4).
\end{equation}
Similarly, the first two (similarly, for the last two) rows of matrix problem~\eqref{eq:mevp} with $K$ and $M$ defined in~\eqref{eq:km2} lead to the dispersion relations
\begin{equation} \label{eq:drep3nep}
  \begin{aligned}
    \lambda^h_j h^2 & = \frac{40 \sin^2(\omega_j h)}{15 + 24 \cos(\omega_j h) + \cos(2 \omega_j h) }, \\
    \lambda^h_j h^2 & = - 20  + \frac{200}{9 + \cos(\omega_j h)}, \\
  \end{aligned}
\end{equation}
which have the error representations  
\begin{equation} \label{eq:dreep2nep}
  \begin{aligned}
    \Lambda^h =  \Lambda + \frac{1}{60} \Lambda^2 + \mathcal{O}(\Lambda^3),\qquad
    \Lambda^h =  \Lambda - \frac{1}{30} \Lambda^2 + \mathcal{O}(\Lambda^3), 
  \end{aligned}
\end{equation}
respectively. These dispersion relations and errors near the boundary are inconsistent with~\eqref{eq:drep2nep} and~\eqref{eq:tep2nep} for the internal nodes. This inconsistency leads to the outliers in the approximate spectrum. The boundary penalization technique proposed in Section~\ref{sec:nep} recovers the consistency weakly. In the limiting case, when the penalty goes to infinity for $C^1$ quadratic elements, we recover consistency strongly.  In the limiting case, when the penalty goes to infinity, the penalized terms in~\eqref{eq:dcvfbfsnep} at the boundaries $x=0$ and $x=1$ become the conditions on the unknown vector components
\begin{equation} \label{eq:3nep}
  U_1 - U_2 = 0, \qquad U_{N+2} - U_{N+1} = 0.
\end{equation}

We apply these conditions to reduce the matrix eigenvalue problem~\eqref{eq:mevp} from dimension $(N+2) \times (N+2)$ to dimension $N\times N$. In particular, we add the first column to the second column, then add the first row to the second row. Similarly, we add the $(N+2)$-th column to the $(N+1)$-th column, then add the $(N+2)$-th row to the $(N+1)$-th row. Removing the first and last rows and columns, we obtain
\begin{equation} \label{eq:km2new}
  \begin{aligned}
    \tilde K & = \frac{1}{h}
    \begin{bmatrix}
      \frac{2}{3} & -\frac{1}{2} & -\frac{1}{6} \\[0.2cm]
      -\frac{1}{2} & 1 & -\frac{1}{3} & -\frac{1}{6} \\[0.2cm]
      -\frac{1}{6} & -\frac{1}{3} & 1 & -\frac{1}{3} & -\frac{1}{6} \\[0.2cm]
      & \ddots & \ddots & \ddots & \ddots & \ddots &  \\[0.2cm]
      & & -\frac{1}{6} & -\frac{1}{3} & 1 & -\frac{1}{3} & -\frac{1}{6} \\[0.2cm]
      & &  & -\frac{1}{6} & -\frac{1}{3} & 1 & -\frac{1}{2}  \\[0.2cm]
      &&& & -\frac{1}{6} & -\frac{1}{2} & \frac{2}{3}  \\
    \end{bmatrix}, \\
    \tilde M &= h
    \begin{bmatrix}
      \frac{23}{30} & \frac{9}{40} & \frac{1}{120} \\[0.2cm]
      \frac{9}{40} & \frac{11}{20} & \frac{13}{60} & \frac{1}{120} \\[0.2cm]
      \frac{1}{120} & \frac{13}{60} & \frac{11}{20} & \frac{13}{60} & \frac{1}{120} \\[0.2cm]
      & \ddots & \ddots & \ddots & \ddots & \ddots &  \\[0.2cm]
      & & \frac{1}{120} & \frac{13}{60} & \frac{11}{20} & \frac{13}{60} & \frac{1}{120} \\[0.2cm]
      & &  & \frac{1}{120} & \frac{13}{60} & \frac{11}{20} & \frac{9}{40} \\[0.2cm]
      &&& & \frac{1}{120} & \frac{9}{40} & \frac{23}{30}  \\
    \end{bmatrix},
  \end{aligned}
\end{equation}
which are of dimensions $N \times N$. Applying Theorem~\ref{thm:evab2}, we have the following result.
\begin{theorem}[Analytical solutions] \label{lem:p2}
  Let $\tilde{K}$ and $\tilde{M}$ be defined in~\eqref{eq:km2new}.  The problem $\tilde K U = \lambda^h \tilde M U$ has analytical eigenpairs $(\lambda_j^h, U_j )$ where
\begin{equation} \label{eq:evefp2nep}
\lambda^h_{j+1} = - 20N^2  + \frac{240N^2 \big(3 + 2 \cos(j \pi h) \big)}{33 + 26 \cos(j \pi  h) + \cos(2 j \pi  h) }, \quad U_{j,k} = c \cos\big(j\pi (k-\frac12) h \big)
\end{equation}
with $h = \frac{1}{N}, k =1,2,\cdots, N, j = 0,1,2, \cdots, N-1, 0 \ne c \in \mathbb{R}$. 
\end{theorem}
 
Similarly, as for the $C^2$ cubic elements for the Dirichlet eigenvalue problem, the corresponding dispersion relations for both the internal and boundary nodes are consistent, and outliers are eliminated. These analytical eigenvector components lead to optimal eigenfunction errors.

\section{Numerical examples} \label{sec:num}

In this section, we present numerical simulations of problem~\eqref{eq:pde} in 1D, 2D, and 3D using isogeometric elements. Once we solve the eigenvalue problem, we sort the discrete eigenvalues in ascending order and pair them with the exact eigenvalues. We focus on the numerical approximation properties of the eigenvalues. In the case of 1D, however, we also report the eigenfunction (EF) errors in $H^1$ semi-norm (energy norm). To have comparable scales, we collect the relative eigenvalue errors and scale the energy norm by the corresponding eigenvalues~\cite{ puzyrev2017dispersion, calo2019dispersion}. The relative eigenvalue errors are 
\begin{equation}
e_{\lambda_j} = \frac{  \lambda^h_j - \lambda_j }{\lambda_j}, \qquad \tilde e_{\lambda_j} = \frac{\tilde \lambda^h_j - \lambda_j }{\lambda_j}.
\end{equation}
From the minimax principle~\cite{ strang1973analysis}, $e_{\lambda_j} >0$.  To see the trend after removing the outliers, we do not apply the absolute values of the eigenvalue errors. We define the scaled eigenfunction error in the energy norm as
\begin{equation}
e_{u_j} = \frac{  |u_j - u_j^h|_{1,\Omega} }{\lambda_j}, \qquad \tilde e_{u_j} = \frac{  |u_j - \tilde u_j^h|_{1,\Omega} }{\lambda_j}.
\end{equation}

\subsection{Dirichlet eigenvalue problem}
\subsubsection{Numerical study in 1D}

We consider $\Omega = [0, 1]$. The one dimensional differential eigenvalue problem~\eqref{eq:pde}  has the following explicit eigenvalues and eigenfunctions 
\begin{equation*}
\lambda_j = j^2 \pi^2, \quad \text{and} \quad u_j = \sqrt{2} \sin( j\pi x), \quad j = 1, 2, \cdots,
\end{equation*}
respectively. 

\begin{figure}[h!]
\centering
\includegraphics[height=8cm]{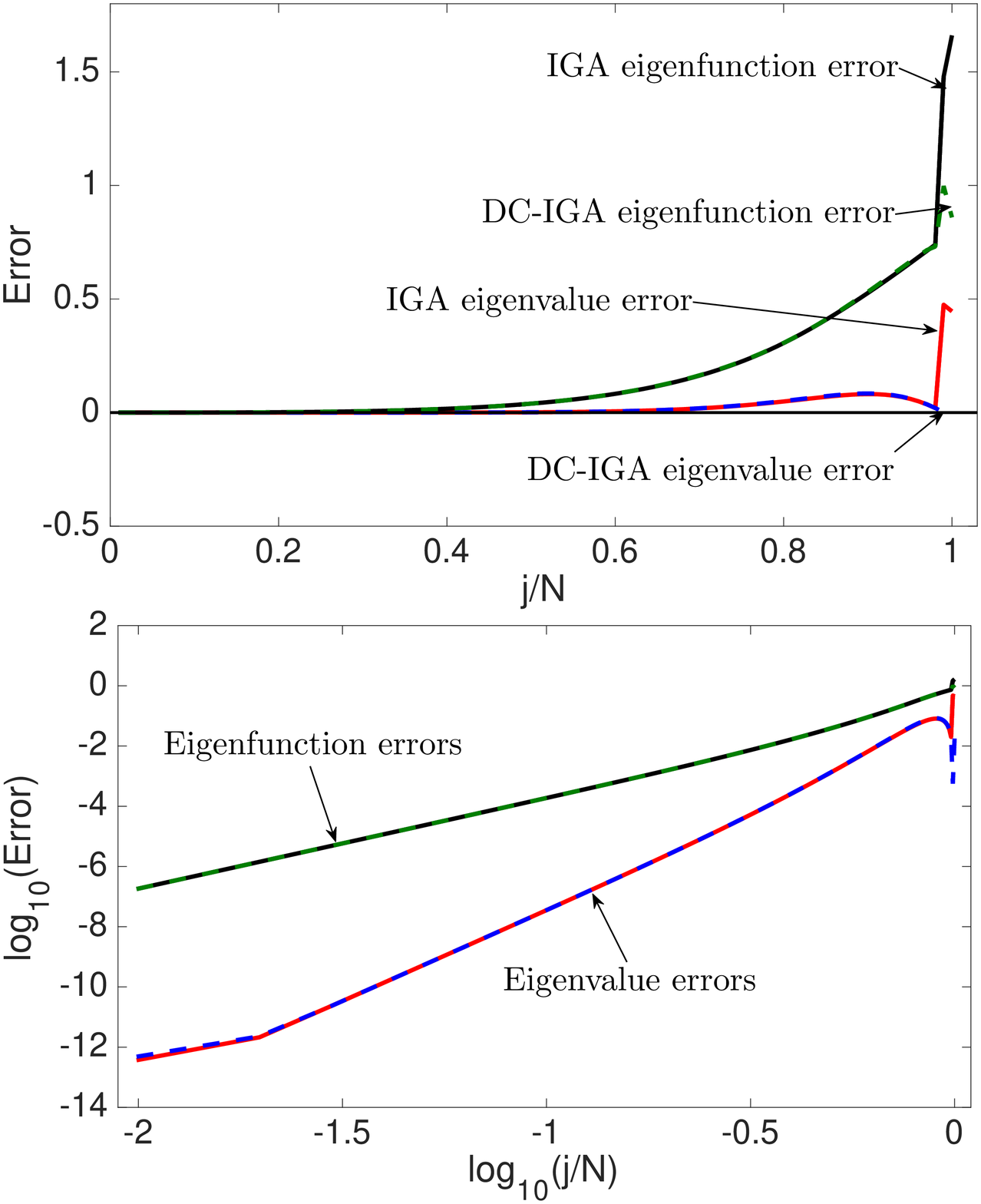} 
\includegraphics[height=8cm]{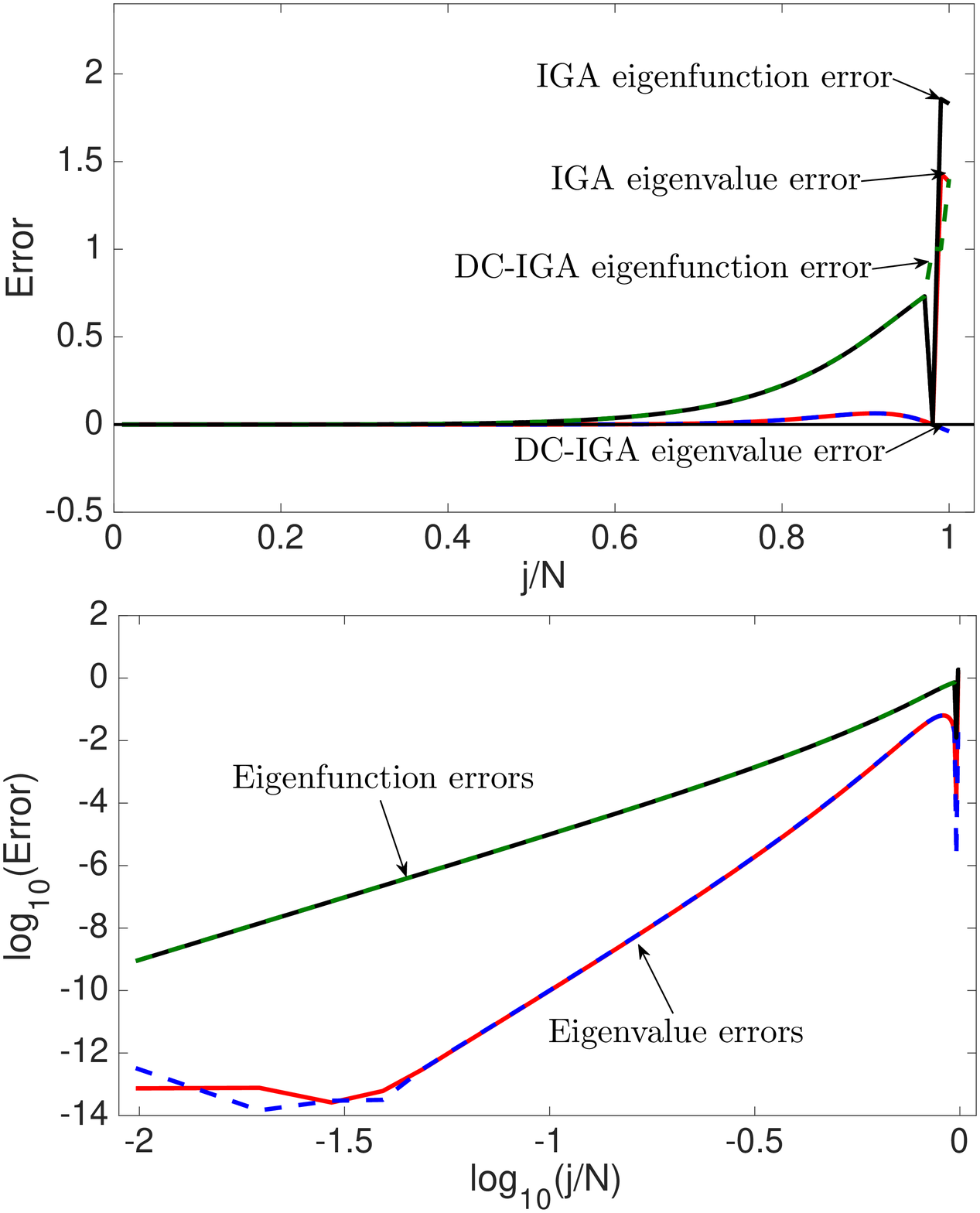} 
\caption{Eigenvalue and eigenfunction errors when using $C^2$ cubic  (left) and  $C^3$ quartic (right) elements with $N=100$ in 1D.}
\label{fig:iga1dp3n100e1}
\end{figure}
\begin{figure}[h!]
\centering
\includegraphics[height=8cm]{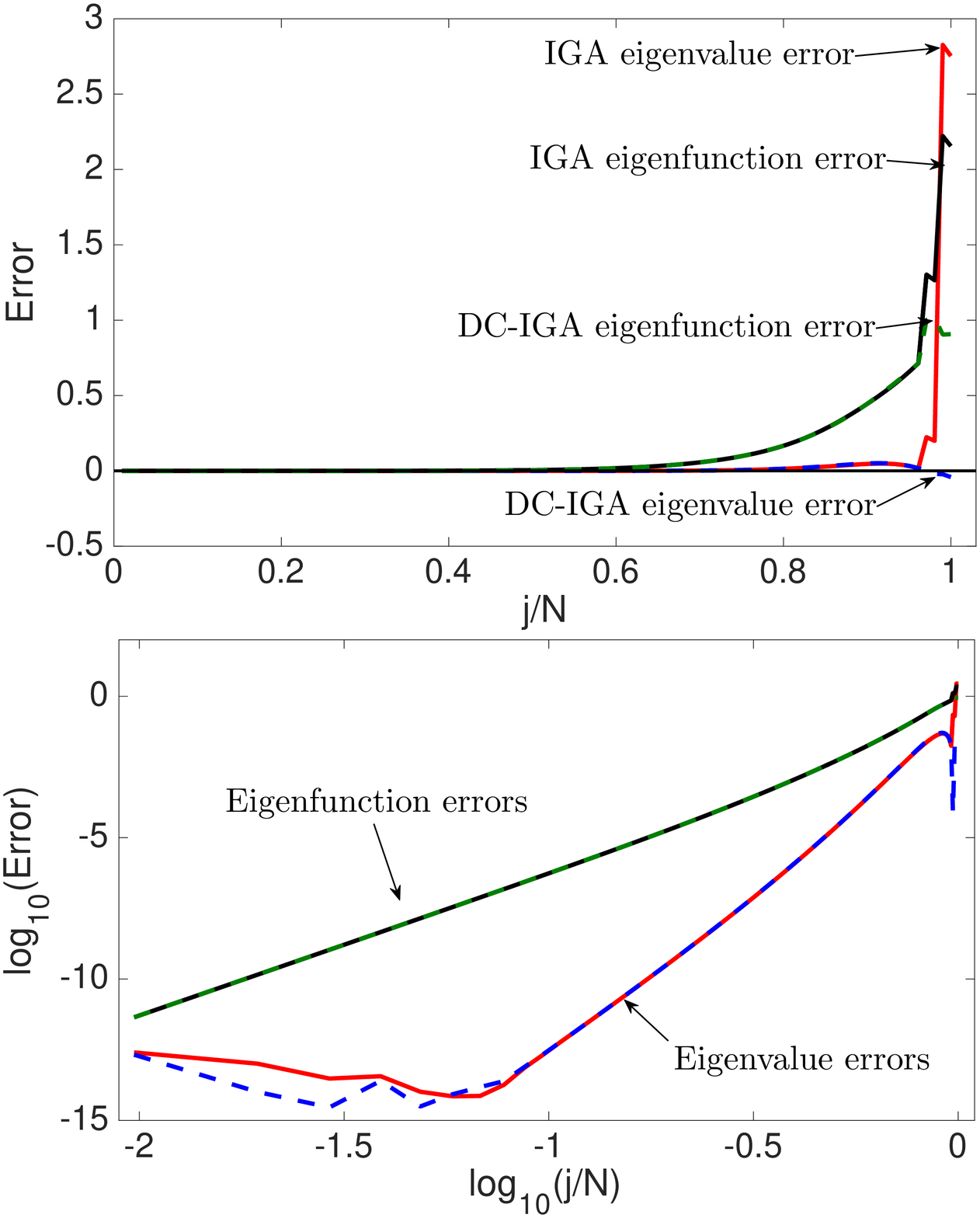} 
\includegraphics[height=8cm]{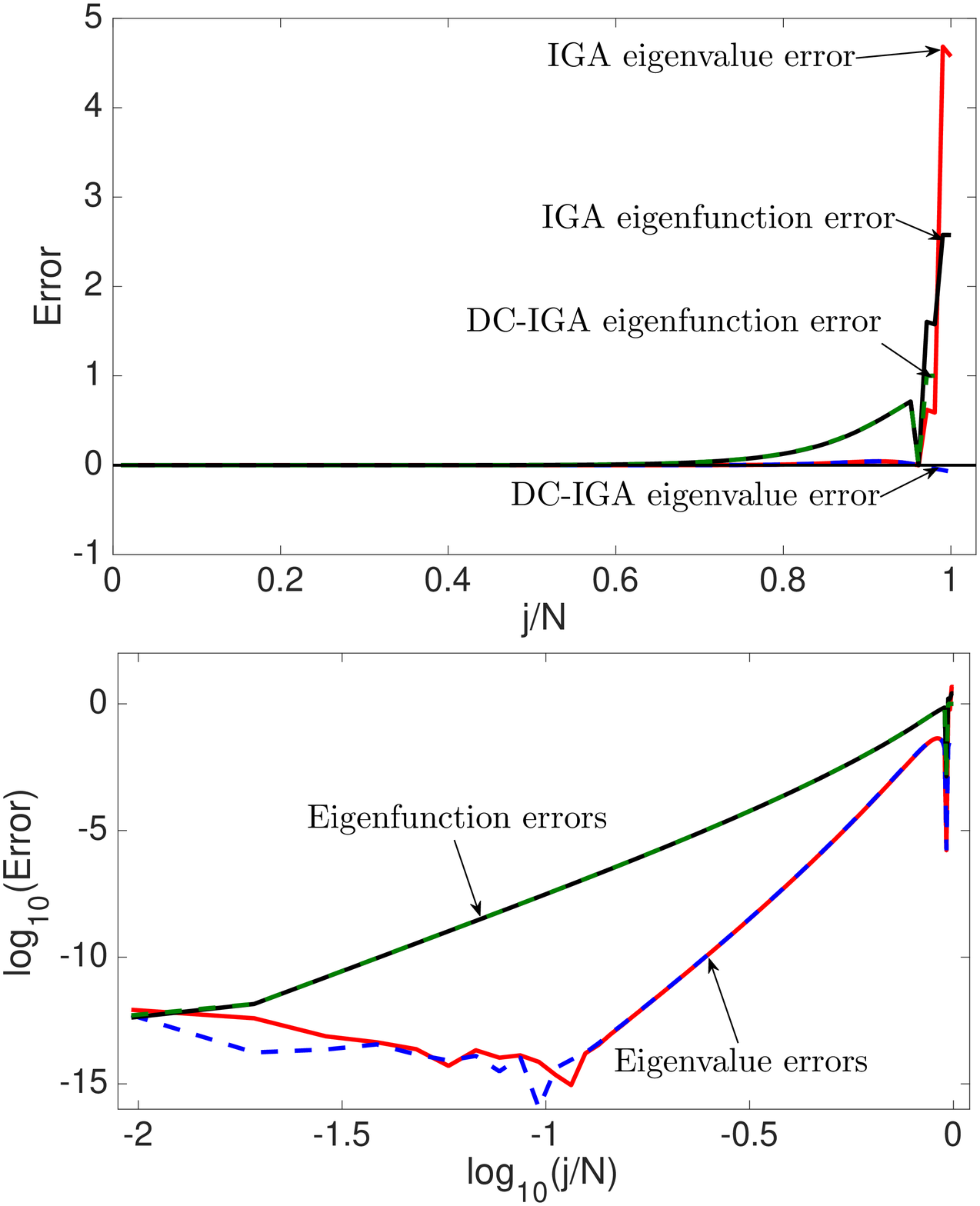} 
\caption{Eigenvalue and eigenfunction errors when using $C^4$ quintic  (left) and $C^5$ sextic (right) elements with $N=100$ in 1D.}
\label{fig:iga1dp5n100e1}
\end{figure}

Figures~\ref{fig:iga1dp3n100e1} and~\ref{fig:iga1dp5n100e1} show the relative eigenvalue and eigenfunction errors in energy norm errors of $C^2$ cubic, $C^3$ quartic, $C^4$ quintic, and $C^5$ sextic isogeometric approximations using the boundary penalization, respectively. In all the cases, we use $N=100$ uniform elements.  We set the penalty parameters to the default value of 1 in all cases. The solid lines show the original error curves of the isogeometric elements, while the dotted lines show the discretely corrected errors. In all the cases, the consistent boundary penalization removes outliers from the spectrum.  Table~\ref{tab:err1d} shows the optimal error convergence rates of approximate eigenvalues and eigenfunctions by the discrete correction. 

For cubic and quartic elements, our experiments show that the boundary penalization with slightly larger penalty parameters (in the scale of $\mathcal{O}(h^{-1})$) also removes the outliers while maintaining accuracy in the lower-frequency region. If the penalty parameters are large (in the scale of $\mathcal{O}(h^{-2})$), then we lose accuracy in the lower-frequency region. We observe similar behavior for quintic and sextic elements.

\begin{table}[ht]
\centering 
\begin{tabular}{| c| c | c  c c || c  c c | ccc|}
\hline
$p$ &  $N$ & $\frac{|\lambda_1^h - \lambda_1|}{\lambda_1}$  & $|u_1^h - u_1|_{H^1}$ & $\|u_1^h - u_1\|_{L^2}$ &  $\frac{|\lambda_6^h - \lambda_6|}{\lambda_6}$  & $|u_6^h - u_6|_{H^1}$ & $\|u_6^h - u_6\|_{L^2}$  \\[0.1cm] \hline

& 8& 	1.31E-07& 	1.14E-03& 	2.31E-05& 	2.99E-02& 	4.06E+00& 	1.29E-01 \\[0.1cm]
& 16& 	1.93E-09& 	1.38E-04& 	1.38E-06& 	1.60E-04& 	2.45E-01& 	2.91E-03 \\[0.1cm]
3 & 32& 	2.98E-11& 	1.71E-05& 	8.48E-08& 	1.63E-06& 	2.42E-02& 	1.27E-04 \\[0.1cm]
& 64& 	5.61E-13& 	2.14E-06& 	5.28E-09& 	2.25E-08& 	2.83E-03& 	7.11E-06 \\[0.1cm] \hline
& $\rho_3$& 	5.95& 	3.02& 	4.03& 	6.76& 	3.48& 	4.7 \\[0.1cm] \hline

& 8& 	1.76E-07& 	1.32E-03& 	6.09E-05& 	1.49E-01& 	1.09E+01& 	4.37E-01 \\[0.1cm]
4& 16	& 3.22E-10& 	5.47E-05& 	1.24E-06& 	4.49E-04& 	4.31E-01& 	1.00E-02 \\[0.1cm]
& 32& 	5.26E-13& 	2.06E-06& 	2.31E-08& 	8.70E-07& 	1.61E-02& 	1.81E-04 \\[0.1cm] \hline
& $\rho_4$& 	9.17& 	4.66& 	5.68& 	8.69& 	4.70& 	5.62 \\[0.1cm] \hline
 
 \end{tabular}
\caption{Relative eigenvalue and eigenfunction errors for $C^2$-cubic and $C^3$-quartic elements in 1D.}
\label{tab:err1d} 
\end{table}

\subsubsection{Numerical study in 2D} \label{sec:num2d}

Let $\Omega = [0, 1] \times [0, 1]$.  The two dimensional differential eigenvalue problem~\eqref{eq:pde} has exact eigenvalues and eigenfunctions 
\begin{equation*}
\lambda_{jk} = ( j^2 + k^2 ) \pi^2, \quad \text{and} \quad u_{jk} = 2 \sin( j\pi x)\sin( k\pi y), \quad j,k = 1, 2, \cdots,
\end{equation*}
respectively. Figures~\ref{fig:iga2dp3n40} and~\ref{fig:iga2dp5n40} show the relative eigenvalue errors of $C^2$ cubic, $C^3$ quartic, $C^4$ quintic, and $C^5$ sextic isogeometric elements, respectively. We build the basis functions on a uniform mesh with $40\times40$ elements.  We set the penalty parameters to 1 as a default. This penalization removes all outliers. More importantly, the corrected relative eigenvalue error curves remain almost flat when compared to the original ones. That is, the overall spectral approximation improves significantly. For brevity, we do not include a full set of 3D results, but the performance is similar; see the left plot in Figure~\ref{fig:iga3d} for an example of $C^3$ quartic isogeometric analysis in 3D. Lastly, the right plot in Figure~\ref{fig:iga3d} shows the comparison of the eigenvalue errors when using a non-uniform mesh with $20\times20\times20$ elements in 3D. The mesh is created using the randomly generated 1D grid with nodes at $x=$ 0, 0.0475, 0.0492, 0.1555, 0.1961, 0.2461, 0.2939, 0.3274, 0.4039, 0.4379, 0.4834, 0.5607, 0.6022, 0.6276, 0.7009, 0.7252, 0.8128, 0.8729, 0.8997, 0.9632, 1.  We observe that the eigenvalue errors in the high-frequency regions are reduced.  

\begin{figure}[h!]
\centering
\includegraphics[height=8cm]{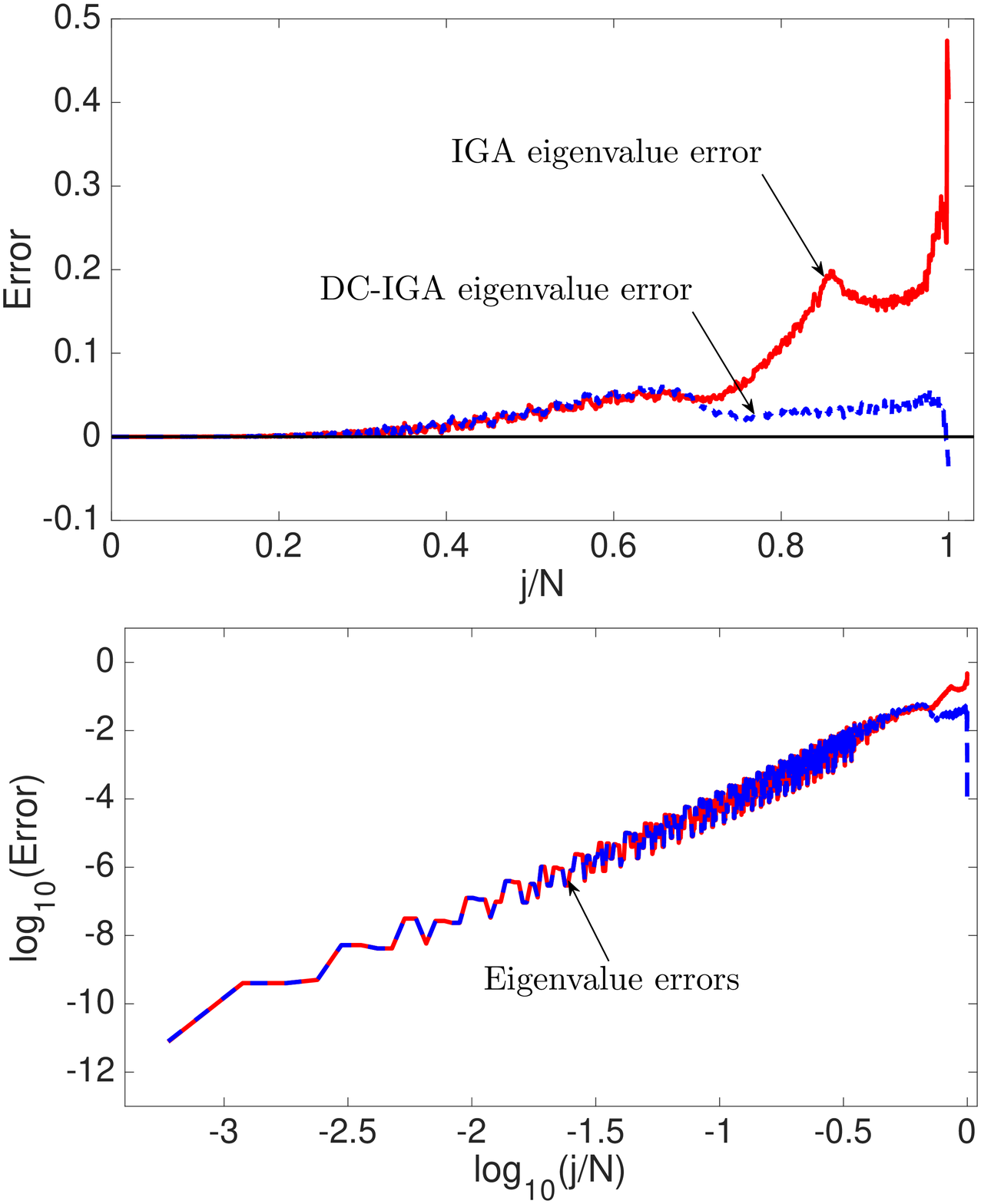} 
\includegraphics[height=8cm]{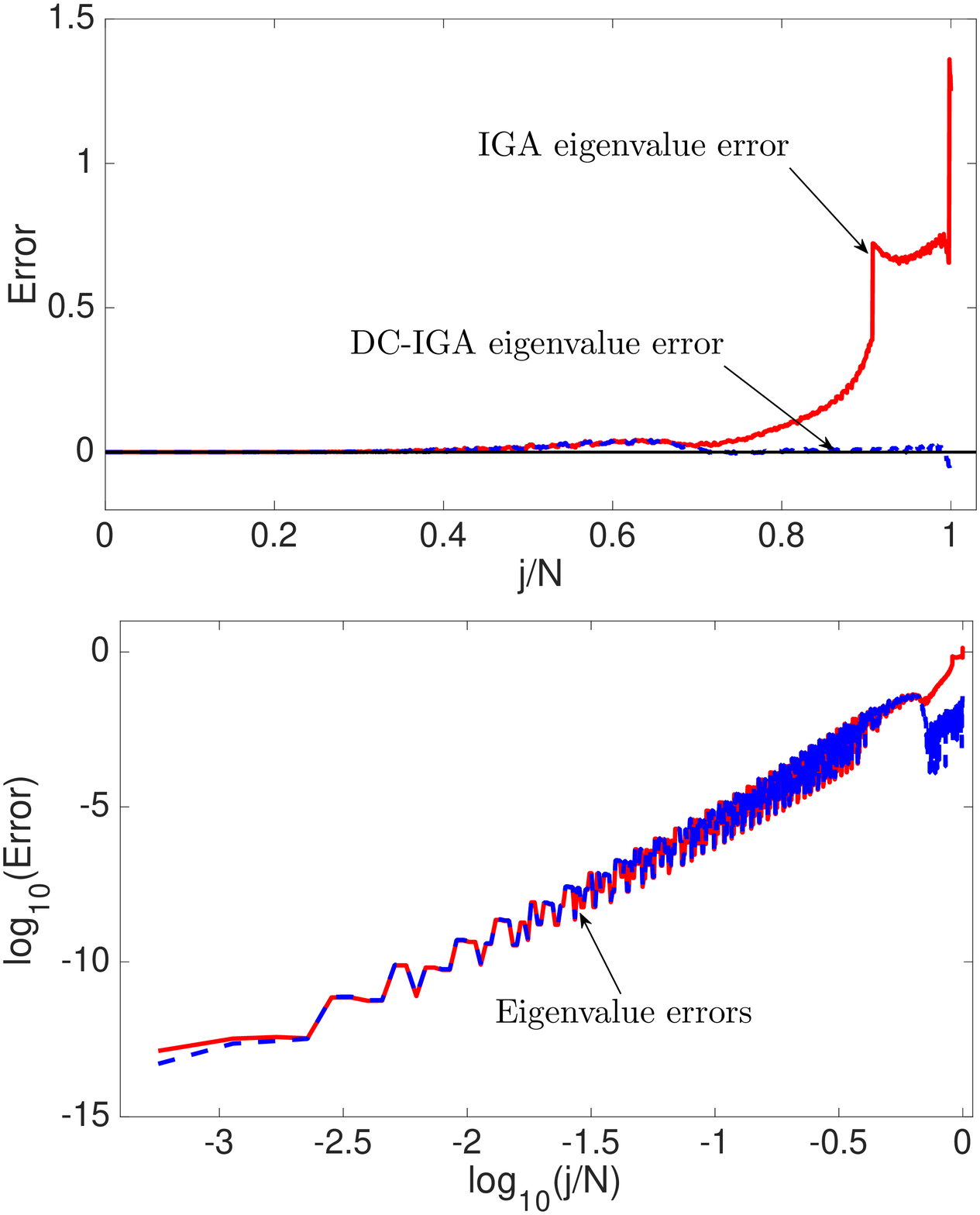} 
\caption{Eigenvalue errors when using $C^2$ cubic  (left) and $C^3$ quartic (right) isogeometric elements with $N=40\times40$ in 2D.}
\label{fig:iga2dp3n40}
\end{figure}

\begin{figure}[h!]
\centering
\includegraphics[height=8cm]{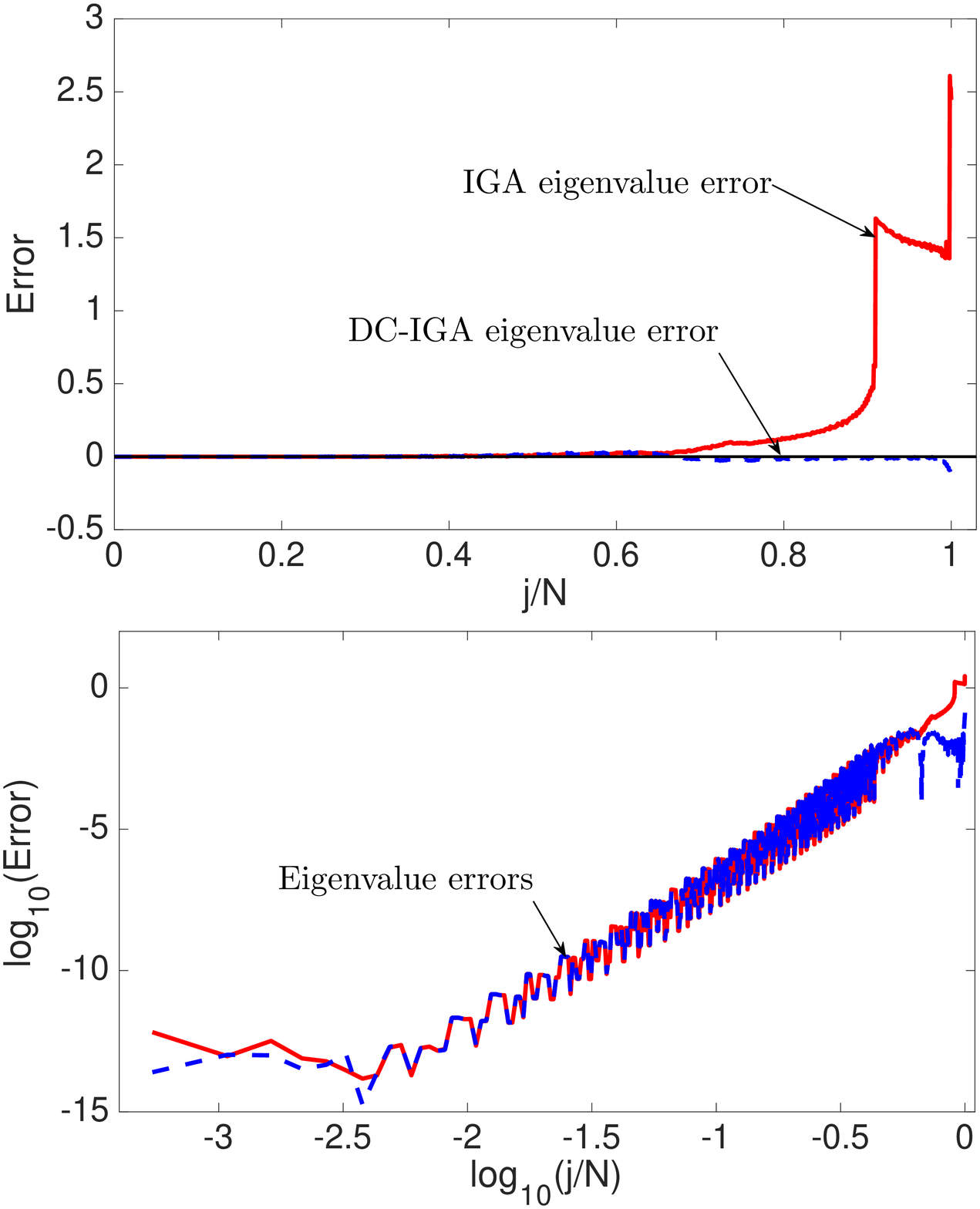} 
\includegraphics[height=8cm]{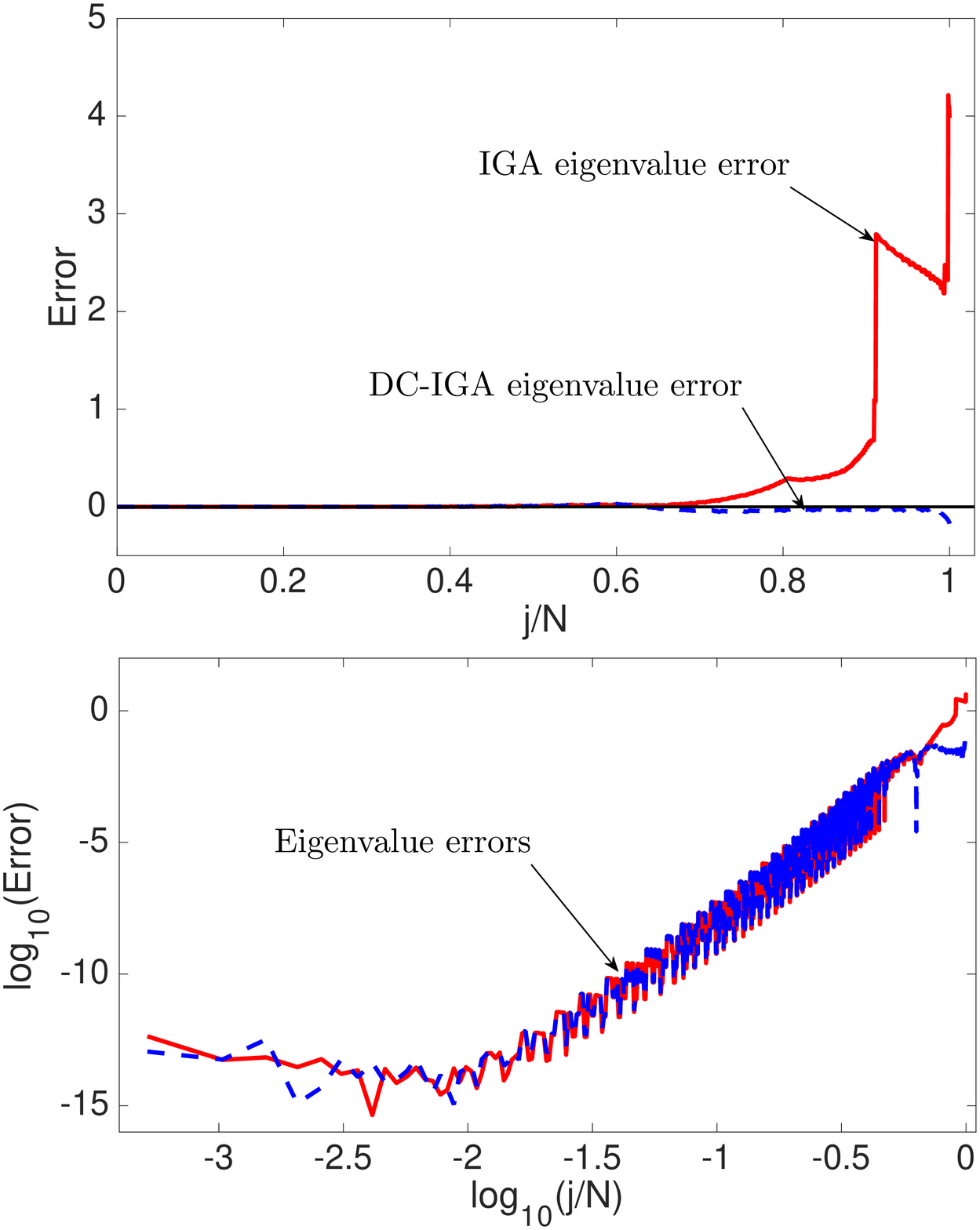} 
\caption{Eigenvalue errors when using $C^4$ quintic  (left) and $C^5$ sextic (right) isogeometric elements with $N=40\times40$ in 2D.}
\label{fig:iga2dp5n40}
\end{figure}

\begin{figure}[h!]
\hspace{-1cm}
\includegraphics[height=6cm]{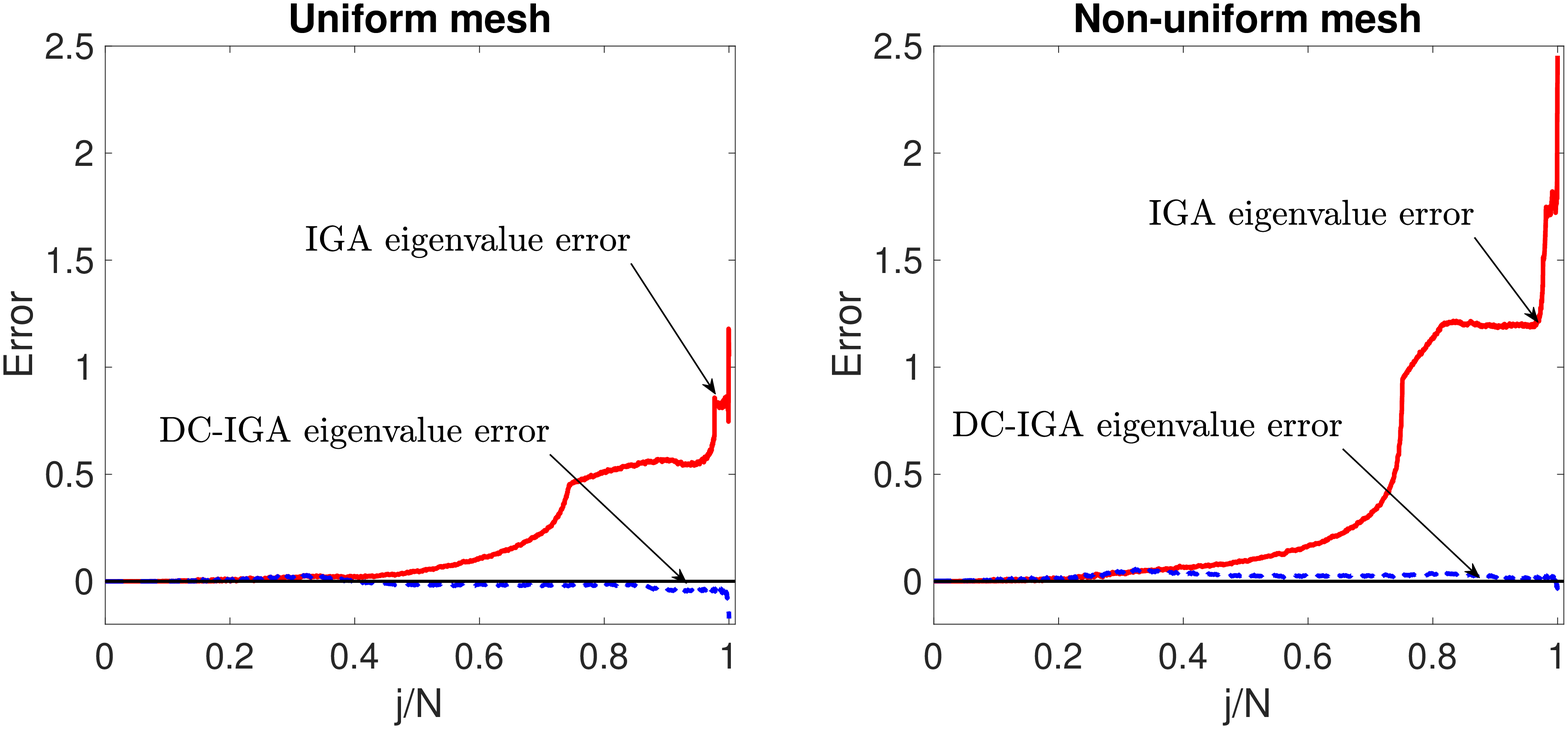} 
\caption{Eigenvalue errors when using $C^3$ quartic isogeometric analysis with $N=20\times20\times20$ uniform elements (left) and non-uniform elements in 3D. The non-uniform grid is generated randomly in 1D and then the 3D mesh is constructed based on the tensor-product.}
\label{fig:iga3d}
\end{figure}

\subsection{Neumann eigenvalue problem} 

For simplicity, we report the results for 2D Neumann eigenvalue problem~\eqref{eq:nep}.  Figures~\ref{fig:iga2dp2n40nep} and~\ref{fig:iga2dp4n40nep} show the relative eigenvalue errors of $C^1$ quadratic, $C^2$ cubic, $C^3$ quartic, and $C^4$ quintic isogeometric elements, respectively. We use a uniform mesh with $40\times40$ elements.  We set the penalty parameters to 1. The boundary penalization removes all the outliers, and the eigenvalue error curves are almost flat compared to the original ones. 

\begin{figure}[h!]
\centering
\includegraphics[height=8cm]{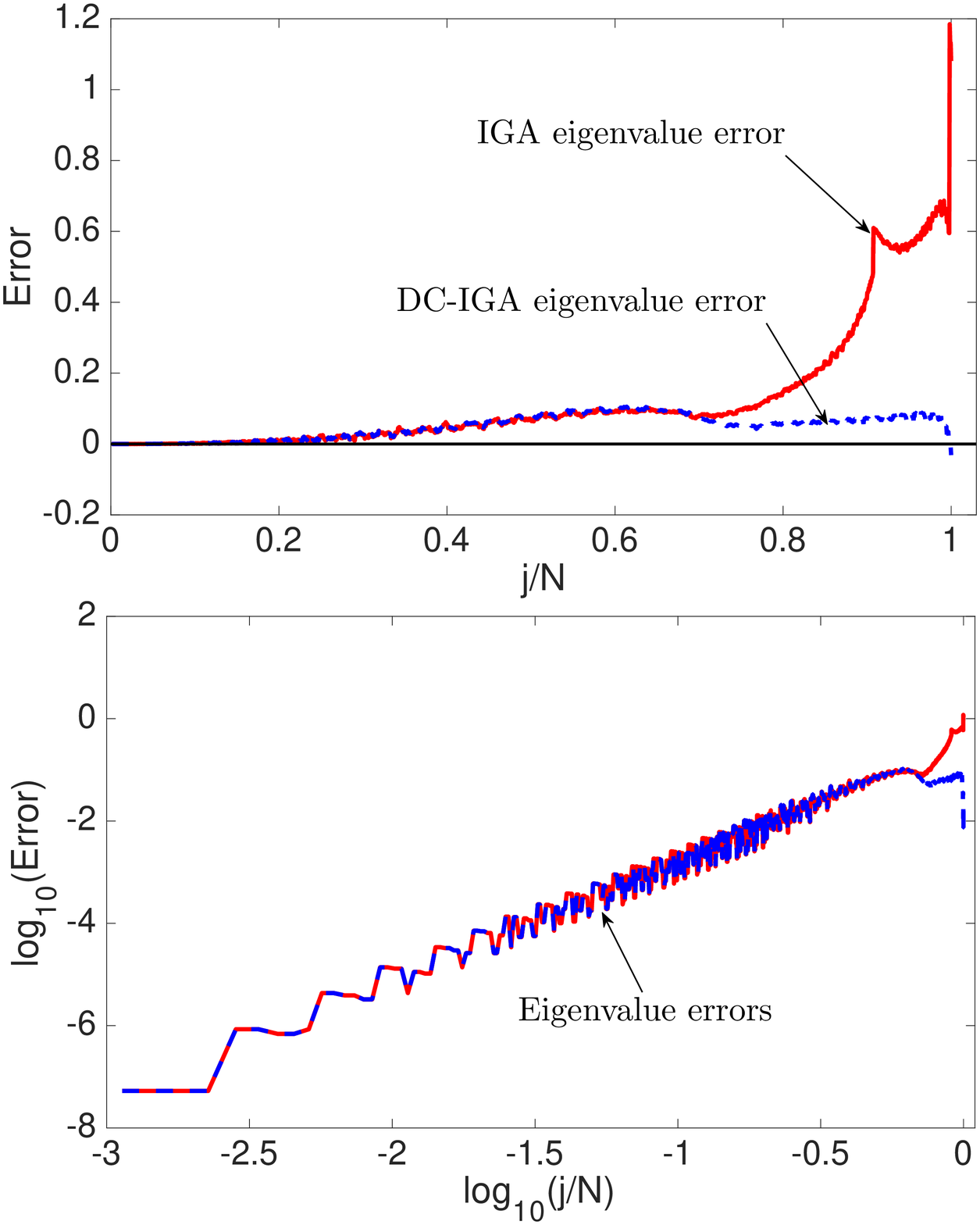} 
\includegraphics[height=8cm]{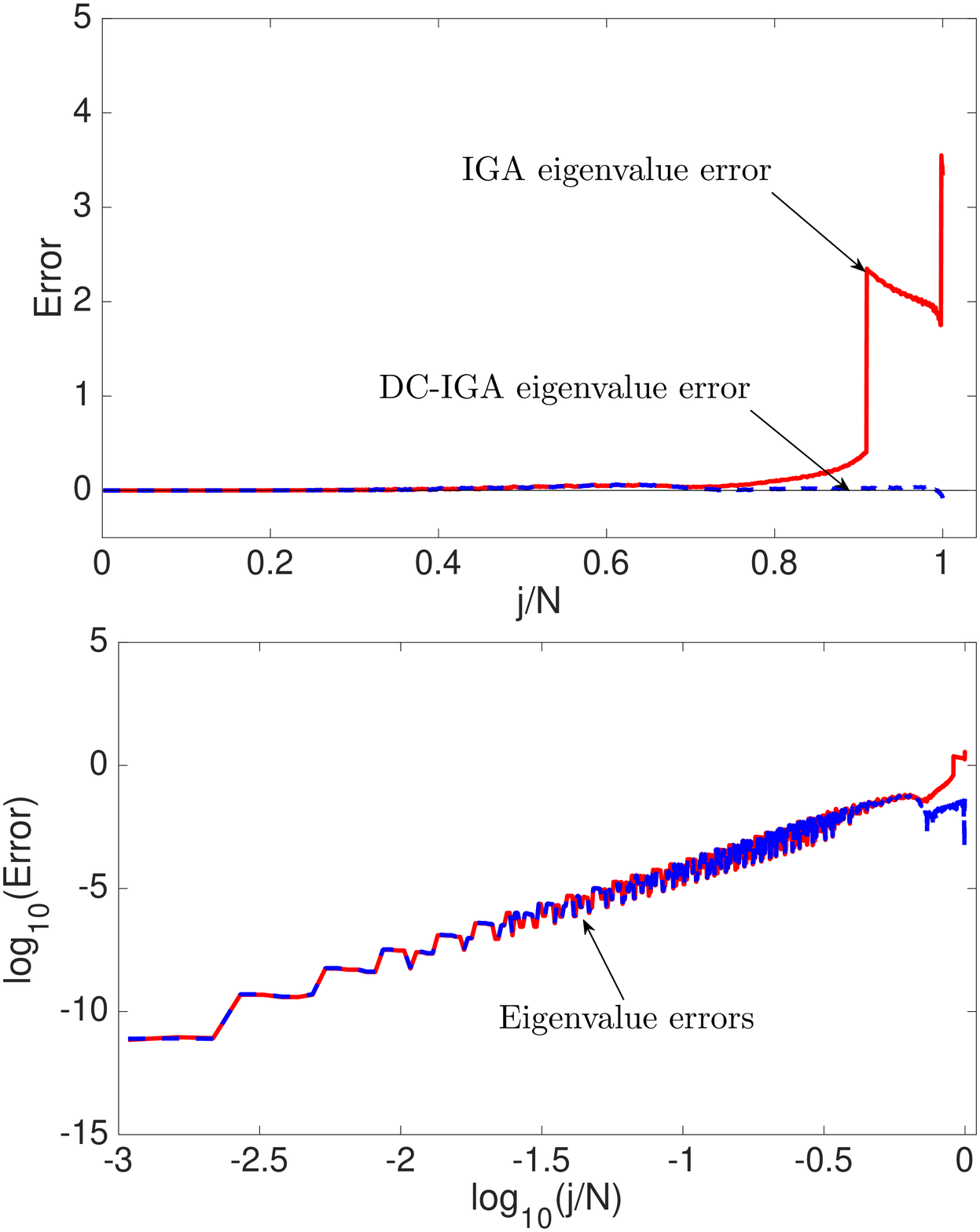} 
\caption{Eigenvalue errors of the Neumann eigenvalue problem when using $C^1$ quadratic   (left) and $C^2$ cubic (right) isogeometric elements with $N=40\times40$ in 2D.}
\label{fig:iga2dp2n40nep}
\end{figure}
\begin{figure}[h!]
\centering
\includegraphics[height=8cm]{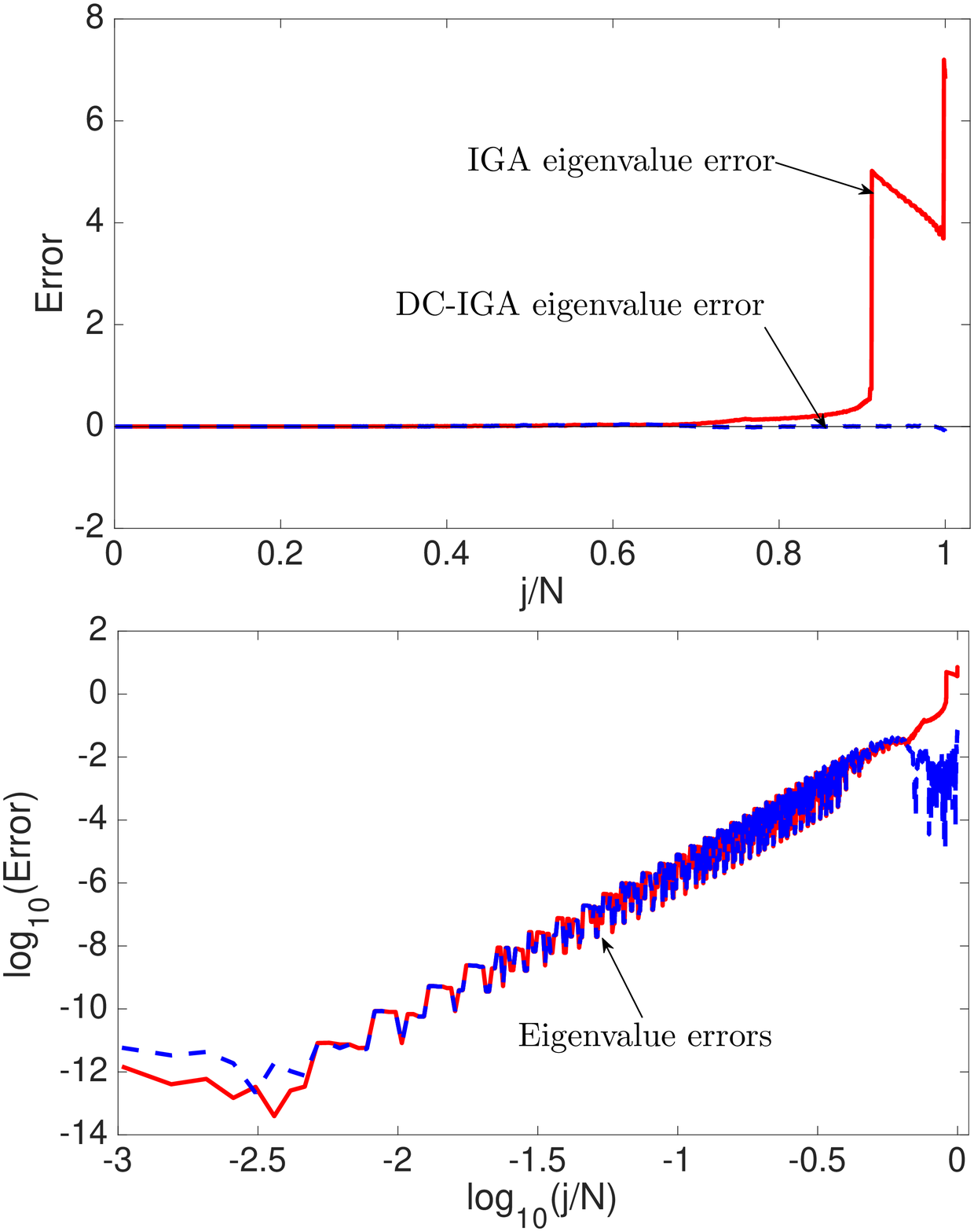} 
\includegraphics[height=8cm]{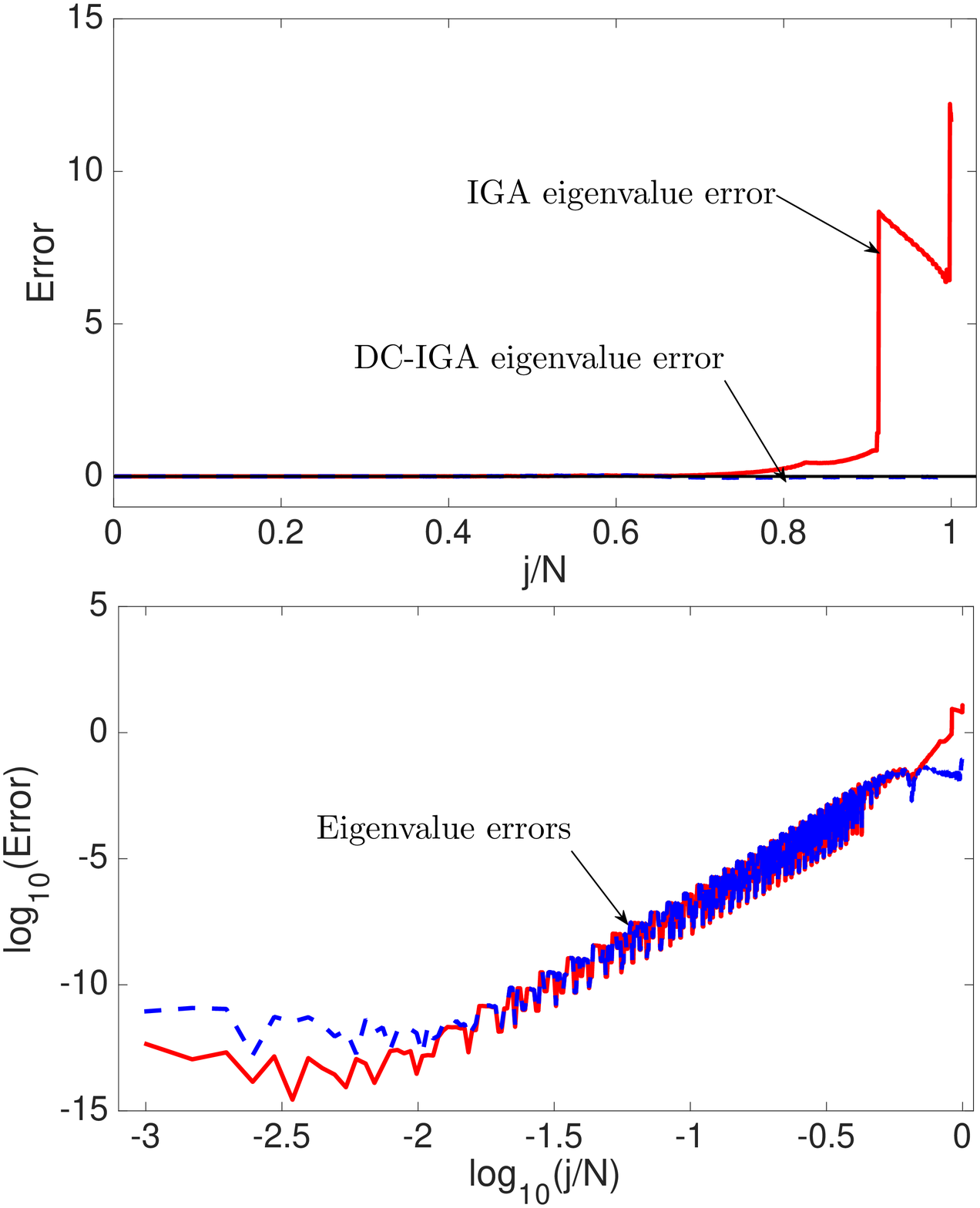} 
\caption{Eigenvalue errors  of the Neumann eigenvalue problem when using $C^3$ quartic  (left) and $C^4$ quintic (right) isogeometric elements with $N=40\times40$ in 2D.}
\label{fig:iga2dp4n40nep}
\end{figure}

\subsection{Numerical study on condition numbers}

We now study the condition numbers to show further advantages of the proposed method.  Due to the symmetry of the stiffness and mass matrices, the condition numbers of the generalized matrix eigenvalue problems~\eqref{eq:mevp} and~\eqref{eq:mevp1d} are given by
\begin{equation}
\gamma := \frac{\lambda^h_{\max}}{\lambda^h_{\min}}, \qquad \tilde \gamma := \frac{\tilde \lambda^h_{\max}}{\tilde \lambda^h_{\min}},
\end{equation}
where $\lambda^h_{\max}, \tilde \lambda^h_{\max}$ are the largest eigenvalues and $\lambda^h_{\min}, \tilde \lambda^h_{\min}$ are the smallest eigenvalues.  The condition number characterizes the stiffness of the system.  We follow the recent work of soft-finite element method (softFEM)~\cite{ deng2020softfem} and define the \textit{condition number reduction ratio} of the method with respect to IGA as
\begin{equation} \label{eq:srr}
\rho := \frac{\gamma}{\tilde\gamma} = \frac{\lambda^h_{\max} }{\tilde \lambda^h_{\max} } \cdot \frac{\tilde \lambda^h_{\min} }{\lambda^h_{\min}}.
\end{equation}
In general, one has $\lambda^h_{\min} \approx \tilde \lambda^h_{\min}$ for IGA and the proposed method with a sufficient number of elements (in practice, these methods with only a few elements already lead to good approximations to the minimum eigenvalues).  Thus, the condition number reduction ratio is mainly characterized by the ratio of the largest eigenvalues.  Finally, we define the \textit{condition number reduction percentage} as
\begin{equation}
\varrho = 100 \frac{\gamma-\tilde \gamma}{\gamma}\,\% = 100(1-\rho^{-1}) \, \%.
\end{equation} 

Table~\ref{tab:cond} shows the smallest and largest eigenvalues, condition numbers and their reduction ratios and percentages for 1D, 2D, and 3D problems.  We observe that the condition numbers of the proposed method are significantly smaller.  The condition number of the proposed method reduces by about 32\% for $C^2$ cubic, 60\% for $C^3$ quartic, 75\% for $C^4$ quintic, and 83\% for $C^5$ sextic elements in 1D, 2D, and 3D.  Similar stiffness reductions are observed in~\cite{ deng2020softfem} for softFEM and in~\cite{ deng2021outlier} for IGA with optimally-blended quadrature rules and discrete corrections.

\begin{table}[ht]
\centering 
\begin{tabular}{|c| c | ccc | ccc c| cc |}
\hline
% & \multicolumn{3}{c|}{first mode} & \multicolumn{3}{c|}{eighth mode} \\[0.1cm] 
$d$ & $p$ & $\lambda_{\min}^h$ &  $\lambda_{\max}^h$ & $ \tilde \lambda_{\max}^h $ &  $\gamma$ &  $\tilde \gamma$ & $\rho$ & $\varrho$ \\[0.1cm] \hline
& 	3& 	9.87& 	5.82E+05& 	3.95E+05& 	5.90E+04& 	4.00E+04& 	1.47& 	32.13\% \\[0.1cm]
1& 	4& 	9.87& 	9.80E+05& 	3.95E+05& 	9.93E+04& 	4.00E+04& 	2.48& 	59.69\% \\[0.1cm]
& 	5& 	9.87& 	1.57E+06& 	4.16E+05& 	1.59E+05& 	4.22E+04& 	3.78& 	73.52\% \\[0.1cm]
& 	6& 	9.87& 	2.38E+06& 	3.99E+05& 	2.41E+05& 	4.05E+04& 	5.96& 	83.22\% \\[0.1cm] \hline
& 	3& 	19.74& 	2.91E+05& 	1.98E+05& 	1.47E+04& 	1.00E+04& 	1.47& 	32.16\% \\[0.1cm]
2& 	4& 	19.74& 	4.90E+05& 	1.97E+05& 	2.48E+04& 	1.00E+04& 	2.48& 	59.69\% \\[0.1cm]
& 	5& 	19.74& 	7.86E+05& 	2.01E+05& 	3.98E+04& 	1.02E+04& 	3.91& 	74.45\% \\[0.1cm]
& 	6& 	19.74& 	1.19E+06& 	1.98E+05& 	6.03E+04& 	1.00E+04& 	6.01& 	83.36\% \\[0.1cm] \hline
& 	3& 	29.61& 	1.09E+05& 	7.41E+04& 	3.69E+03& 	2.50E+03& 	1.47& 	32.16\% \\[0.1cm]
3& 	4& 	29.61& 	1.84E+05& 	7.40E+04& 	6.20E+03& 	2.50E+03& 	2.48& 	59.69\% \\[0.1cm]
& 	5& 	29.61& 	2.95E+05& 	7.44E+04& 	9.95E+03& 	2.51E+03& 	3.96& 	74.76\% \\[0.1cm]
& 	6& 	29.61& 	4.46E+05& 	7.41E+04& 	1.51E+04& 	2.50E+03& 	6.02& 	83.40\% \\[0.1cm] \hline
 \end{tabular}
\caption{Minimal and maximal eigenvalues, condition numbers, reduction ratios and percentages when using IGA and DC-IGA. The polynomial degrees are $p\in\{3,4,5,6\}$. There are 200, $100\times 100$,  and $50 \times 50 \times 50$ elements in 1D, 2D, and 3D, respectively.}
\label{tab:cond} 
\end{table}

\section{Concluding remarks} \label{sec:conclusion} 

We remove the outliers from the discrete spectrum of isogeometric analysis. The idea is to penalize the boundary elements with extra information from the higher-order derivatives. By removing the outliers, we improve the overall spectrum significantly, especially in multiple dimensions. Consequently, we improve the approximate eigenfunctions in the high-frequency regions. The methods add penalizations to the discrete systems only for the degrees of freedom near the boundaries; thus, the boundary penalization implementation is simple to perform in existing isogeometric analysis codes. Our numerical simulation also shows that the technique works equally well for both the Dirichlet and Neumann eigenvalue problems when using the open knots and non-open knots. However, for curved boundary and domains with geometrical mappings, a simple direct implementation does not improve much the high-frequency eigenvalue errors, and further study on the boundary penalty terms is necessary.

Future work will analyze the dispersion error for higher-order elements to characterize the approximate eigenvalue errors analytically and remove the outliers by imposing strong boundary conditions, which reconstructs the basis functions near the boundary elements.  We will also analyze the impact of the outlier-removal techniques on time marching for time-dependent partial differential equations, especially the effects of high-frequency error pollutions in the time-marching schemes.  From the study on stiffness reduction in Section 5.3, we expect to have a larger critical time step allowing for faster explicit in time isogeometric analysis of engineering problems.

\section*{Acknowledgments}

The authors thank professor Thomas J.R. Hughes (UT Austin) for several discussions and comments which led to an improvement of the paper.  This publication was also made possible in part by the CSIRO Professorial Chair in Computational Geoscience at Curtin University and the Deep Earth Imaging Enterprise Future Science Platforms of the Commonwealth Scientific Industrial Research Organisation, CSIRO, of Australia. This project has received funding from the European Union's Horizon 2020 research and innovation program under the Marie Sklodowska-Curie grant agreement No 777778 (MATHROCKS). The Curtin Corrosion Centre and the Curtin Institute for Computation kindly provide ongoing support.
 %Q. Deng thanks the support from the department of mathematics 

%\section*{References}

%\bibliographystyle{plain}
%\bibliographystyle{siam}
%\bibliography{ref.bib}
%\bibliographystyle{siamplain}
%\bibliographystyle{siam}
\bibliography{igaref}

%\appendix{}
%\section{Other posibilities} 
%We present other 

\end{document}